\input ./style/arxiv-vmsta.cfg
\documentclass[numbers,compress,v1.0.1]{vmsta}
\usepackage{vtexbibtags}
\usepackage{mathbh,amsmath}
\usepackage{enumerate}

\volume{6}
\issue{4}
\pubyear{2019}
\firstpage{443}
\lastpage{478}
\aid{VMSTA143}
\doi{10.15559/19-VMSTA143}
\articletype{research-article}

\DeclareMathOperator*{\plim}{P-lim}
\DeclareMathOperator{\dist}{dist}
\DeclareMathOperator{\Cov}{Cov}
\DeclareMathOperator{\sgn}{sgn}
\DeclareMathOperator{\err}{err}
\DeclareMathOperator\supp{supp}

\startlocaldefs
\ifdefined\HCode 
\def\index#1{}
\else 
\fi

\theoremstyle{definition}
\newtheorem{rem}{Remark}[section]
\newtheorem{expl}[rem]{Example}
\newtheorem{definition}[rem]{Definition}
\newtheorem{ass}[rem]{Assumption}

\theoremstyle{plain}
\newtheorem{prop}[rem]{Proposition}
\newtheorem{lem}[rem]{Lemma}
\newtheorem{theo}[rem]{Theorem}
\newtheorem{cor}[rem]{Corollary}

\numberwithin{equation}{section}

\allowdisplaybreaks
\endlocaldefs

\begin{document}

\begin{frontmatter}
\pretitle{Research Article}

\title{On a linear functional for infinitely divisible moving average random fields}

\author{\inits{S.}\fnms{Stefan}~\snm{Roth}\ead[label=e1]{stefan.roth@alumni.uni-ulm.com}}
\address{\institution{Institute of Stochastics}, Helmholtzstra\ss e 18, 89081 Ulm, \cny{Germany}}



\markboth{S. Roth}{On a linear functional for infinitely divisible moving average random fields}

\begin{abstract}
Given a low-frequency sample of the infinitely divisible moving average
random field $\{ \int _{\mathbb{R}^{d}} f(t-x) \Lambda (dx), \ t
\in \mathbb{R}^{d} \}$, in~\cite{GlRothSpo017} we proposed an
estimator $\widehat{uv_{0}}$ for the function $\mathbb{R}\ni x \mapsto
u(x)v_{0}(x) = (uv_{0})(x)$, with $u(x) = x$ and $v_{0}$ being the
L\'{e}vy density of the integrator random measure $\Lambda $. In this
paper, we study asymptotic properties of the linear functional
$L^{2}(\mathbb{R}) \ni v \mapsto \left \langle v, \widehat{uv_{0}} \right \rangle _{L^{2}(\mathbb{R})}$, if the (known) kernel function $f$ has a compact
support. We provide conditions that ensure consistency (in mean) and
prove a central limit theorem for it.
\end{abstract}

\received{\sday{25} \smonth{10} \syear{2018}}
\revised{\sday{23} \smonth{6} \syear{2019}}
\accepted{\sday{13} \smonth{9} \syear{2019}}
\publishedonline{\sday{22} \smonth{10} \syear{2019}}

\end{frontmatter}

\section{Introduction}%
\label{sect:Int}
Consider a stationary infinitely divisible indepently scattered random
measure\index{random ! measure} $\Lambda $ whose L\'{e}vy density\index{L\'{e}vy ! density} is denoted by $v_{0}$. For
some (known) $\Lambda $-integrable function $f:\mathbb{R}^{d} \to
\mathbb{R}$ with a compact support, let
%
\begin{equation}
\label{eq:moving_average_rf}
X = \{ X(t); \ t \in \mathbb{R}^{d} \}, \qquad X(t)=\int _{\mathbb{R}
^{d}} f(t-x) \Lambda (dx)
\end{equation}
be the corresponding moving average random field.\index{random ! field} In our recent preprint
\cite{GlRothSpo017}, we proposed an estimator $\widehat{uv_{0}}$
for the function $\mathbb{R}\ni x \mapsto u(x)v_{0}(x) = (uv_{0})(x)$,
$u(x) = x$, based on low frequency observations $(X(j \Delta ))_{j
\in W}$ of $X$, with $\Delta > 0$ and $W$ a finite subset of
$\mathbb{Z}^{d}$.

A wide class of spatio-temporal processes with the spectral
representation~\eqref{eq:moving_average_rf} is provided by the so-called
ambit random fields,\index{random ! field} where a space-time L\'{e}vy process\index{L\'{e}vy ! process} serves as
integrator. Such processes are, e.g., used to model the growth rate of
tumours, where the spatial component describes the angle between the
center of the tumour cell and the nearest point at its boundary
(cf.~\cite{BarnNielsSchmiegel07,jonsdottir2008}).
Ambit fields cover quite a number of different processes and fields
including Ornstein--Uhlenbeck type and mixed moving average random fields\index{random ! field}
(cf.~\cite{BarnNiels11,BarnNielsSchmiegel04}). A
further interesting application of~\eqref{eq:moving_average_rf} is given
in~\cite{Karcher12}, where the author uses infinitely divisible
moving average random fields\index{random ! field} in order to model claims of natural
disaster insurance within different postal code areas.

We point out that there is a large number of literature concerning
estimation of the L\'{e}vy density $v_{1}$\index{L\'{e}vy ! density} (its L\'{e}vy measure,
respectively) in the case when $X$ is a L\'{e}vy process\index{L\'{e}vy ! process}
(cf.~\cite{belo2010,comte,comte1,gugushvili,neumann}).
Moreover, in the recent paper~\cite{BelPanWoern16} the authors
provide an estimator for the L\'{e}vy density $v_{0}$\index{L\'{e}vy ! density} of the integrator
L\'{e}vy process\index{L\'{e}vy ! process} $\{L_{s}\}$ of a moving average process $X(t) =
\int _{\mathbb{R}} f(t-s) dL_{s}$, $t \in \mathbb{R}$, which covers the
case $d=1$ in \eqref{eq:moving_average_rf}. For a discussion on the
differences between our approach and the method provided
in~\cite{BelPanWoern16}, we refer to~\cite{GlRothSpo017}
and~\cite{KaRoSpoWalk18}.

In this paper, we investigate asymptotic properties of the linear
functional $L^{2}(\mathbb{R}) \ni v \mapsto \hat{\mathcal{L}}_{W} v =
\left \langle v, \widehat{uv_{0}} \right \rangle _{L^{2}(\mathbb{R})}$
as the sample size $|W|$ tends to infinity. It is motivated by the paper
of Nickl and Reiss~\cite{nickl}, where the authors provide a
Donsker type theorem for the L\'{e}vy measure of pure jump L\'{e}vy
processes.\index{L\'{e}vy ! process} Since our observations are $m$-dependent, the classical
i.i.d. theory does not apply here. Instead, we combine results of Chen
and Shao~\cite{Chen2004} for $m$-dependent random fields and ideas
of Bulinski and Shashkin~\cite{Bulinski07} with exponential
inequalities for weakly dependent random fields\index{random ! field} (see e.g.
\cite{Heinrich,dedecker}) in order to prove our limit
theorems.

It turns out that under certain regularity assumptions on $uv_{0}$,
$\hat{\mathcal{L}}_{W} v$ is a mean consistent estimator for
$\mathcal{L}v = \left \langle v, uv_{0}\right \rangle _{L^{2}(
\mathbb{R})}$ with a rate of convergence given by
$\mathcal{O}(|W|^{-1/2})$, for any $v$ that belongs to a subspace
$\mathcal{U}$ of $L^{1}(\mathbb{R}) \cap L^{2}(\mathbb{R})$. Moreover,
we give conditions such that finite dimensional distributions of the
process $\{ |W|^{1/2} (\hat{\mathcal{L}}_{W} - \mathcal{L})v; \ v
\in \mathcal{U}\}$ are asymptotically Gaussian as $|W|$ is regularly
growing\index{regularly growing} to infinity.

From a practical point of view, a naturally arising question is wether
a proposed model for $v_{0}$ (or equivalently $uv_{0}$) is suitable.
Knowledge of the asymptotic distribution of $|W|^{1/2} (\hat{\mathcal{L}}
_{W} - \mathcal{L})$ can be used in order to construct tests for
different hypotheses, e.g., on regularity assumptions of the model for
$v_{0}$. Indeed,
the scalar
product $\left \langle \, \cdot \,, \, \cdot \,\right \rangle _{L^{2}(
\mathbb{R})}$ naturally induces that the class $\mathcal{U}$ of test functions is
growing, when $uv_{0}$ becomes more regular.

This paper is organized as follows. In Section~\ref{sec:preliminaries},
we give a brief overview of regularly growing\index{regularly growing} sets and infinitely
divisible moving average random fields.\index{random ! field} We further recall some notation
and the most frequently used results from~\cite{GlRothSpo017}.
Section~\ref{section:clt} is devoted to asymptotic properties of
$\hat{\mathcal{L}}_{W}$. Here we discuss our regularity assumptions and
state the main results of this paper
(Theorems~\ref{cor:order_of_convergence} and~\ref{theo:clt_univariate}).
Section~\ref{proof:clt_univariate} is dedicated to the proofs of our
limit theorems. Some of the shorter proofs as well as external results
that will frequently be used in Section~\ref{section:clt} are moved to
Appendix.

\section{Preliminaries}%
\label{sec:preliminaries}

\subsection{Notation}%
\label{subsec:notation}
Throughout this paper, we use the following notation.

By $\mathcal{B}(\mathbb{R}^{d})$ we denote the Borel $\sigma $-field on
the Euclidean space $\mathbb{R}^{d}$. The Lebesgue measure on
$\mathbb{R}^{d}$ is denoted by $\nu _{d}$ and we shortly write
$\nu _{d}(dx) = dx$ when we integrate w.r.t. $\nu _{d}$. For any
measurable space $(M, \mathcal{M}, \mu )$ we denote by $L^{\alpha }(M)$,
$1 \leq \alpha < \infty $, the space of all $\mathcal{M}|\mathcal{B}(
\mathbb{R})$-\xch{measurable}{mesurable} functions $f:M \rightarrow \mathbb{R}$ with
$\int _{M} |f|^{\alpha }(x) \mu (dx) < \infty $. Equipped with the norm
$||f||_{L^{\alpha }(M)} = \left ( \int _{M} |f|^{\alpha }(x) \mu (dx) \right )
^{1/\alpha }$, $L^{\alpha }(M)$ becomes a Banach space and, in the
case $\alpha =2$, even a Hilbert space with scalar product $\left \langle f,g
\right \rangle _{L^{\alpha }(M)}= \int _{M} f(x)g(x)\mu (dx)$, for any
$f,g \in L^{2}(M)$. With $L^{\infty }(M)$ (i.e. if $\alpha = \infty $)
we denote the space of all real-valued bounded functions on $M$. In the case
$(M, \mathcal{M}, \mu ) = (\mathbb{R}, \mathcal{B}(\mathbb{R}), \nu
_{1})$ we denote by
\begin{equation*}
H^{\delta }(\mathbb{R}) =\Bigl \{ f \in L^{2}(\mathbb{R}): \ \int _{
\mathbb{R}} |\mathcal{F}_{+} f|^{2} (x)(1+x^{2})^{\delta } dx <\infty
\Bigr\}
\end{equation*}
the Sobolev space of order $\delta > 0$ equipped with the Sobolev norm
$||f||_{H^{\delta }(\mathbb{R})} = ||\mathcal{F}_{+} f(\cdot ) (1+
\cdot ^{2})^{\delta /2}||_{L^{2}(\mathbb{R})}$, where $\mathcal{F}_{+}$
is the Fourier transform on $L^{2}(\mathbb{R})$. For $f \in L^{1}(
\mathbb{R})$, $\mathcal{F}_{+} f$ is defined by $\mathcal{F}_{+} f (x)
= \int _{\mathbb{R}}e^{itx}f(t)dt$, $x \in \mathbb{R}$. Throughout the
rest of this paper $(\Omega , \mathcal{A}, \mathbb{P})$ denotes a
probability space. Note that in this case $L^{\alpha }(\Omega )$ is the
space of all random variables with finite $\alpha $-th moment. For an
arbitrary set $A$ we introduce 
the notation $
\textup{card}(A)$ or briefly $|A|$ for its cardinality. Let
$\supp(f)=\{ x\in \mathbb{R}^{d}: f(x)\neq 0\}$ be the support
set of a function $f: \mathbb{R}^{d}\to \mathbb{R}$. Denote by
$\textup{diam}(A)=\sup \{ \| x-y \|_{\infty }: x,y\in A \}$ the diameter
of a bounded set $A\subset \mathbb{R}^{d}$.

\subsection{Regularly growing\index{regularly growing} sets}%
\label{subsec:regularly_growing_sets}
In this \xch{section}{secion}, we briefly recall some basic facts about \textit{regularly
growing sets}. For a more detailed investigation on this topic, see,
e.g.,~\cite{Bulinski07}.

Let $a = (a_{1}, \dots , a_{d}) \in \mathbb{R}^{d}$ be a vector with
positive components. In the sequel, we shortly write $a > 0$ in this
situation. Moreover, let
\begin{equation*}
\Pi _{0}(a) = \{ x \in \mathbb{R}^{d}, \ 0 < x_{i} \leq a_{i}, \ i=1,
\dots ,d \}
\end{equation*}
and define for any $j \in \mathbb{Z}^{d}$ the \textit{shifted block}
$\Pi _{j}(a)$ by
\begin{equation*}
\Pi _{j}(a) = \Pi _{0}(a) + j a = \{ x \in \mathbb{R}^{d}, \ j_{i} a
_{i} < x_{i} \leq j_{i} (a_{i} + 1), \ i=1,\dots ,d \}.
\end{equation*}
Clearly $\{ \Pi _{j}, \ j \in \mathbb{Z}^{d} \}$ forms a partition of
$\mathbb{R}^{d}$. For any $U \subset \mathbb{Z}^{d}$, introduce the sets
\begin{equation*}
\begin{split}
J_{-}(U, a) = \{ j \in \mathbb{Z}^{d}, \ \Pi _{j}(a) \subset U \},
\qquad
& J_{+}(U, a) = \{ j \in \mathbb{Z}^{d}, \ \Pi _{j}(a) \cap U
\neq \emptyset \}
\\
U^{-}(a) = \bigcup \limits _{j\in J_{-}(U, a)} \Pi _{j}(a), \qquad
& U
^{+}(a) = \bigcup \limits _{j\in J_{+}(U, a)} \Pi _{j}(a).
\end{split}
\end{equation*}
A sequence of sets $U_{n} \subset \mathbb{R}^{d}$ ($n \in \mathbb{N}$)
\textit{tends to infinity in Van Hove sense} or shortly is
\textit{VH-growing}, if for any $a > 0$
\begin{equation*}
\nu _{d}(U_{n}^{-}) \to \infty \quad \text{and} \quad \frac{\nu _{d}(U
_{n}^{-})}{\nu _{d}(U_{n}^{+})} \to 1 \quad \text{as} \ n \to \infty .
\end{equation*}
For a finite set $A \subset \mathbb{Z}^{d}$, define by $\partial A =
\{ j \in \mathbb{Z}^{d} \backslash A, \ \dist (j,A) = 1 \}$ its
boundary, where $\dist (j,A) = \inf \{ \|j-x\|_{\infty }, \ x \in A
\}$.

A sequence of finite sets $A_{n} \in \mathbb{Z}^{d}$ ($n \in
\mathbb{N}$) is called \textit{regularly growing\index{regularly growing} (to infinity)}, if
\begin{equation*}
|A_{n}| \to \infty , \quad \text{and} \quad \frac{|\partial A_{n}|}{|A
_{n}|} \to 0, \quad \text{as} \ n \to \infty .
\end{equation*}
%
\begin{rem}
Regular growth of a family $A_{n} \subset \mathbb{Z}^{d}$ means that the
number of points in the boundary of $A_{n}$ grows significantly slower
than the number of its interior points.
\end{rem}
The following result that connects regularly and VH-growing sequences
can be found in~\cite[p.174]{Bulinski07}.
%
\begin{lem}
\label{lem:VH_growing_regularly_growing}
\begin{enumerate}%
\item
Let $U_{n} \subset \mathbb{R}^{d}$ ($n \in \mathbb{N}$) be VH-growing.
Then $V_{n} = U_{n} \cap \mathbb{Z}^{d}$ ($n \in \mathbb{N}$) is
regularly growing\index{regularly growing} to infinity.
\item
If $(U_{n})_{n\in \mathbb{N}}$ is a sequence of finite subsets of
$\mathbb{Z}^{d}$, regularly growing\index{regularly growing} to infinity, then $V_{n} =
\cup _{j \in U_{n}} [j,j+1)$ is VH-\xch{growing}{grwoing}, where $[j,j+1) = \{ x
\in \mathbb{R}^{d}: \ j_{k} \leq x_{k} < j_{k} + 1, \ k=1,\dots ,d \}$.
\end{enumerate}
\end{lem}

\subsection{Infinitely divisible random measures\index{infinitely divisible random measure}}%
\label{subsec:reminder}
In what follows, denote by $\mathcal{E}_{0}(\mathbb{R}^{d})$ the collection
of all bounded Borel sets in $\mathbb{R}^{d}$.

Suppose that $\Lambda = \{\Lambda (A); \ A \in \mathcal{E}_{0}(\mathbb{R}
^{d})\}$ is an infinitely divisible random measure\index{infinitely divisible random measure} on some
probability space $(\Omega , \mathcal{A}, P)$, i.e. a random measure\index{random ! measure}
with the following properties:
\begin{enumerate}%
\item[(a)] Let $(E_{m})_{m\in \mathbb{N}}$ be a sequence of disjoint
sets in $\mathcal{E}_{0}(\mathbb{R}^{d})$. Then the sequence
$(\Lambda (E_{m}))_{m\in \mathbb{N}}$ consists of independent random
variables; if, in addition,\break  $\cup _{m=1}^{\infty }E_{m} \in
\mathcal{E}_{0}(\mathbb{R}^{d})$, then we have $\Lambda (\cup _{m=1}
^{\infty }E_{m}) =\sum _{m=1}^{\infty }\Lambda (E_{m})$ almost surely.
\item[(b)] The random variable $\Lambda (A)$ has an infinitely divisible
distribution for any choice of $A \in \mathcal{E}_{0}(\mathbb{R}^{d})$.
\end{enumerate}

For every $A\in \mathcal{E}_{0}(\mathbb{R}^{d})$, let $
\varphi _{\Lambda (A)}$ denote the characteristic function\index{characteristic function} of the random
variable $\Lambda (A)$. Due to the infinite divisibility of the random
variable $\Lambda (A)$, the characteristic function\index{characteristic function} $\varphi _{\Lambda
(A)}$ has a L\'{e}vy--Khintchin representation which can, in its most
general form, be found in \cite[p. 456]{Rajput}. Throughout the
rest of the paper we make the additional assumption that the
L\'{e}vy--Khintchin representation of $\Lambda (A)$ is of a special form,
namely
\begin{equation*}
\varphi _{\Lambda (A)}(t) = \exp \left \lbrace \nu _{d}(A) K(t)
\right \rbrace , \quad A \in \mathcal{E}_{0}(\mathbb{R}^{d}),
\end{equation*}
with
%
\begin{equation}
\label{eq:K}
K(t) = ita_{0} - \frac{1}{2} t^{2} b_{0} + \int \limits _{\mathbb{R}}
\left ( e^{itx} - 1 - itx \Eins_{[-1,1]}(x) \right )v_{0}(x)dx,
\end{equation}
where $\nu _{d}$ denotes the Lebesgue measure on $\mathbb{R}^{d}$,
$a_{0}$ and $b_{0}$ are real numbers with $0 \leq b_{0} < \infty $ and
$v_{0}: \mathbb{R}\to \mathbb{R}$ is a L\'{e}vy density,\index{L\'{e}vy ! density} i.e. a
measurable function which \xch{fulfills}{fulfils} $\int _{\mathbb{R}} \min \{1,x^{2}\}v
_{0}(x)dx < \infty $. The triplet $(a_{0},b_{0},v_{0})$ will be referred
to as the \emph{L\'{e}vy characteristic\index{L\'{e}vy ! characteristic}} of $\Lambda $. It uniquely
determines the distribution of $\Lambda $. This particular structure of\vadjust{\goodbreak}
the characteristic functions\index{characteristic function} $\varphi _{\Lambda (A)}$ means that the
random measure $\Lambda $\index{random ! measure} is stationary with the control measure
$\lambda : \mathcal{B}(\mathbb{\mathbb{R}}) \rightarrow [0,\infty )$
given by
\begin{equation*}
\lambda (A) = \nu _{d}(A) \left [ |a_{0}| + b_{0} + \int
\limits _{\mathbb{R}} \min \{1,x^{2}\} v_{0}(x)dx \right ]
\quad
\text{for all } A \in \mathcal{E}_{0}(\mathbb{R}^{d}).
\end{equation*}

Now one can define the stochastic integral with respect to the
infinitely divisible random measure $\Lambda $\index{infinitely divisible random measure} in the following way:
\begin{enumerate}%
\item
Let $f = \sum _{j=1}^{n} x_{j} \Eins_{A_{j}}$ be a real simple
function on $\mathbb{R}^{d}$, where $A_{j} \in \mathcal{E}_{0}(
\mathbb{R}^{d})$ are pairwise disjoint. Then for every $A \in
\mathcal{B}(\mathbb{R}^{d})$ we define
\begin{equation*}
\int \limits _{A}f(x)\Lambda (dx) = \sum \limits _{j=1}^{n} x_{j} \Lambda
(A \cap A_{j}).
\end{equation*}%
\item
A measurable function $f:(\mathbb{R}^{d},\mathcal{B}(\mathbb{R}^{d}))
\rightarrow (\mathbb{R}, \mathcal{B}(\mathbb{R}))$ is said to be
$\Lambda $-inte\-grable if there exists a sequence $(f^{(m)})_{m \in
\mathbb{N}}$ of simple functions as in (1) such that $f^{(m)} \rightarrow
f$ holds $\lambda $-almost everywhere and such that, for each
$A \in \mathcal{B}(\mathbb{R}^{d})$, the sequence $\left ( \int _{A} f
^{(m)}(x)\Lambda (dx) \right )_{m \in \mathbb{N}}$ converges in
probability as $m \rightarrow \infty $. In this case we set
\begin{equation*}
\int \limits _{A} f(x) \Lambda (dx) = \plim\limits _{m\rightarrow \infty
} \int \limits _{A} f^{(m)}(x)\Lambda (dx).
\end{equation*}
\end{enumerate}

A useful characterization of the $\Lambda $-integrability of a function
$f$ is given in \cite[Theorem 2.7]{Rajput}. Now let $f:
\mathbb{R}^{d} \rightarrow \mathbb{R}$ be $\Lambda $-integrable; then
the function $f(t - \cdot )$ is $\Lambda $-integrable for every
$t\in \mathbb{R}^{d}$ as well. We define the moving average random field
$X = \{X(t), \ t \in \mathbb{R}^{d}\}$ by
%
\begin{equation}
\label{eq:moving_average}
X(t) = \int \limits _{\mathbb{R}^{d}} f(t-x)\Lambda (dx), \quad t
\in \mathbb{R}^{d}.
\end{equation}
Recall that a random field\index{random ! field} is called \emph{infinitely divisible} if its
finite dimensional distributions are infinitely divisible. The random
field\index{random ! field} $X$ above is (strictly) stationary and infinitely divisible and
the characteristic function $\varphi _{X(0)}$\index{characteristic function} of $X(0)$ is given by
\begin{equation*}
\varphi _{X(0)}(u) = \exp \left ( \int _{\mathbb{R}^{d}} K(uf(s)) \:
\mathrm{d}s \right ),
\end{equation*}
where $K$ is the function from~\eqref{eq:K}. The argument $
\int _{\mathbb{R}^{d}} K(uf(s)) \:\mathrm{d}s$ in the above exponential
function can be shown to have a similar structure as $K(t)$; more
precisely, we have
%
\begin{equation}
\label{eq:exponent_psi}
\int _{\mathbb{R}^{d}} K(uf(s)) \:\mathrm{d}s = i u a_{1} -
\frac{1}{2}u^{2} b_{1} + \int \limits _{\mathbb{R}} \left (e^{iux}-1-iux
\Eins_{[-1,1]}(x)\right )v_{1}(x) \:\mathrm{d}x
\end{equation}
where $a_{1}$ and $b_{1}$ are real numbers with $b_{0} \geq 0$ and the
function $v_{1}$ is the L\'{e}vy density\index{L\'{e}vy ! density} of $X(0)$. The triplet
$(a_{1},b_{1},v_{1})$ is again referred to as the
\emph{L\'{e}vy characteristic\index{L\'{e}vy ! characteristic}} (of $X(0)$) and determines the
distribution of $X(0)$ uniquely. A simple computation shows that the
triplet $(a_{1}, b_{1}, v_{1})$ is given by the formulas
%
\begin{align}
& a_{1} = \int \limits _{\mathbb{R}^{d}}U(f(s)) \:\mathrm{d}s,
\qquad
b_{1} = b_{0} \int \limits _{\mathbb{R}^{d}} f^{2}(s) \:\mathrm{d}s,
\nonumber
\\
& v_{1}(x) = \int \limits _{\supp(f)} \frac{1}{|f(s)|}v_{0}
\left ( \frac{x}{f(s)} \right ) \:\mathrm{d}s ,
\label{eq:levy-characterisitic-of-field}
\end{align}
where $\supp(f) := \{s \in \mathbb{R}^{d}: \ f(s) \neq 0\}$
denotes the support of $f$ and the function $U$ is defined via
\begin{equation*}
U(u) = u \left ( a_{0} + \int _{\mathbb{R}} x \left [
\Eins_{[-1,1]}(ux)- \Eins_{[-1,1]}(x) \right ] v_{0}(x)\:
\mathrm{d}x \right ).
\end{equation*}
Note that the $\Lambda $-integrability of $f$ immediately implies that
$f \in L^{1}(\mathbb{R}^{d}) \cap L^{2}(\mathbb{R}^{d})$. Hence, all
integrals above are finite.

For details on the theory of infinitely divisible measures and fields
we refer the interested reader to \cite{Rajput}.

\subsection{A plug-in estimation approach for $\mathbf{v_{0}}$}%
\label{subsec:plug-in}
Let the random field $X = \{X(t), \ t\in \mathbb{R}^{d} \}$ be given as
in Section~\ref{subsec:reminder} and define the function $u:
\mathbb{R}\to \mathbb{R}$ by $u(x) = x$. Suppose further that an estimator
$\widehat{uv_{1}}$ for $uv_{1}$ is given. In our recent
preprint~\cite{GlRothSpo017}, we provided an estimation approach
for $uv_{0}$ based on relation~\eqref{eq:levy-characterisitic-of-field}
which we briefly recall in this section. Therefore, quite a number of
notations are required.

Assume that $f$ satisfies the integrability condition
%
\begin{equation}
\label{eq:integrability_condition_f}
\int _{\supp(f(s))} |f(s)|^{1/2} ds < \xch{\infty ,}{\infty .}
\end{equation}
and define the operator $\mathcal{G}: L^{2}(\mathbb{R}) \to L^{2}(
\mathbb{R})$ by
\begin{equation*}
\mathcal{G}v = \int _{\supp(f)} \sgn (f(s)) v \Big ( \frac{\,
\cdot \,}{f(s)} \Big ) ds, \quad v \in L^{2}(\mathbb{R}).
\end{equation*}
Moreover, define the isometry $\mathcal{M}: L^{2}(\mathbb{R}) \to L
^{2}(\mathbb{R}^{\times }, \frac{dx}{|x|})$ by
\begin{equation*}
(\mathcal{M}v)(x) = |x|^{1/2} v(x), \quad v \in L^{2}(\mathbb{R}),
\end{equation*}
and let the functions $m_{f,\pm }:\mathbb{R}^{\times }\to \mathbb{C}$
and $\mu _{f}:\mathbb{R}^{\times }\to \mathbb{C}$ be given by
\begin{equation*}
\begin{split}
m_{f,+}(x)
& = \int _{\supp(f)} \sgn (f(s)) |f(s)|^{1/2} e^{-i
x \log |f(s)|} ds,
\\
m_{f,-}(x)
& = \int _{\supp(f)} |f(s)|^{1/2} e^{-i x
\log |f(s)|} ds,
\\
\mu _{f}(y)
& =
\begin{cases}
m_{f,+}(\log |y|)
& \text{if} \ y > 0,
\\
m_{f,-}(\log |y|)
& \text{if} \ y < 0.
\end{cases}
\end{split}
\end{equation*}
Multiplying both sides in~\eqref{eq:levy-characterisitic-of-field} by
$u$ leads to the equivalent relation
%
\begin{equation}
\label{eq:operator_equation}
uv_{1} = \mathcal{G}uv_{0}.
\end{equation}
Suppose $uv_{1} \in L^{2}(\mathbb{R})$ and assume that for some
$\beta \geq 0$,
%
\begin{equation}
\label{eq:existence_uniqueness_condition}
|m_{f,\pm }(x)| \gtrsim \frac{1}{1 + |x|^{\beta }}, \quad
\text{for all} \ x \in \mathbb{R}\tag{$\mathbf{U}_{\beta }$}.
\end{equation}
Then the unique solution $uv_{0} \in L^{2}(\mathbb{R})$ to
equation~\eqref{eq:operator_equation} is given by
\begin{equation*}
uv_{0} = \mathcal{G}^{-1} uv_{1} = \mathcal{M}^{-1} \mathcal{F}_{
\times }^{-1} \Big ( \frac{1}{\mu _{f}} \mathcal{F}_{\times }
\mathcal{M}uv_{1} \Big ),
\end{equation*}
cf.~\cite[Theorem 3.1]{GlRothSpo017}. Based on this relation, the
paper~\cite{GlRothSpo017} provides the estimator
%
\begin{equation}
\label{eq:Gi_n_inverse}
\widehat{uv_{0}} = \mathcal{M}^{-1} \mathcal{F}_{\times }^{-1} \Big ( \frac{1}{
\mu _{f,n}} \mathcal{F}_{\times }\mathcal{M}\widehat{uv_{1}} \Big ) =:
\mathcal{G}_{n}^{-1} \widehat{uv_{1}}
\end{equation}
for $uv_{0}$, where $(a_{n})_{n\in \mathbb{N}} \subseteq (0,\infty )$
is an arbitrary sequence, depending on the sample size $n$, that
tends to $0$ as $n \to \infty $, and the mapping $
\frac{1}{\mu _{f,n}}: \mathbb{R}\to \mathbb{C}$ is defined by
$\frac{1}{\mu _{f,n}} := \frac{1}{\mu _{f}} \Eins _{ \{ |\mu _{f}| >
a_{n} \} }$. Here $\mathcal{F}_{\times }:L^{2}(\mathbb{R}^{\times },
\frac{dx}{|x|}) \to L^{2}(\mathbb{R}^{\times }, \frac{dx}{|x|})$ denotes
the Fourier transform on the multiplicative group $\mathbb{R}^{\times
}$ which is defined by
\begin{equation*}
(\mathcal{F}_{\times }u)(y) = \int _{\mathbb{R}^{\times }} u(x) \; e
^{-i\log \lvert x\rvert \cdot \log \lvert y \rvert } \cdot e^{i\pi
\delta (x) \delta (y)} \, \frac{\mathrm{d}x}{\lvert x\rvert },
\end{equation*}
for all $u \in L^{1}(\mathbb{R}^{\times }, \frac{dx}{|x|}) \cap L^{2}(
\mathbb{R}^{\times }, \frac{dx}{|x|})$, with $\delta :\mathbb{R}^{
\times }\to \mathbb{R}$ given by $\delta (x)
=\Eins_{(-\infty ,0)}(x)$ (cf.~\cite[Section
2.2]{GlRothSpo017}). A more detailed introduction to harmonic analysis
on locally compact abelian groups can be found, e.g.,
in~\cite{deitmar2009}.
%
\begin{rem}
The linear operator $\mathcal{G}_{n}^{-1} : L^{2}(\mathbb{R}) \to L
^{2}(\mathbb{R})$ defined in~\eqref{eq:Gi_n_inverse} is bounded in the
operator norm $\|\mathcal{G}_{n}^{-1}\| \leq \frac{1}{a_{n}}$, whereas
$\mathcal{G}^{-1}$ is unbounded in general.
\end{rem}

\subsection{$m$-dependent random fields}

A random field\index{random ! field} $X = \{ X(t), \ t \in T \}$, $T \subseteq \mathbb{R}
^{d}$, defined on some probability space $(\Omega , \mathcal{A},
\mathbb{P})$ is called $m$-\textit{dependent} if for some $m \in
\mathbb{N}$ and any finite subsets $U$ and $V$ of $T$ the random vectors
$(X(u))_{u \in U}$ and $(X(v))_{v \in V}$ are independent whenever
\begin{equation*}
\|u-v\|_{\infty }= \max \{ |u_{i} - v_{i}|, \ i=1,\dots ,d \} > m,
\end{equation*}
for all $u=(u_{1},\dots ,u_{d})^{\top }\in U$ and $v = (v_{1},\dots ,v
_{d})^{\top }\in V$.
%
\begin{lem}
\label{lem:m_depend_X}
Let the random field $X$ be given in~\eqref{eq:moving_average} and
suppose that $f$ has a compact support. Then $X$ is $m$-dependent with
$m > \textup{diam}(\supp(f))$.
\end{lem}
\begin{proof}
Compactness of $\supp(f)$ implies that $\supp(f(t-
\cdot ))$ and $\supp(f(s-\cdot ))$ are disjoint whenever
$\|t-s\|_{\infty }> \textup{diam}(\supp(f))$. Since further
$\Lambda $ is independently scattered and integration
in~\eqref{eq:moving_average} is done only on $\supp(f(t-
\cdot ))$, $X(t)$ and $X(s)$ are independent for $\|t-s\|_{\infty }>
\textup{diam}(\supp(f))$.
\end{proof}

\section{A linear functional for infinitely divisible moving averages}%
\label{section:clt}

\subsection{The setting}

Let $\Lambda = \{ \Lambda (A), \ A \in \mathcal{E}_{0}(\mathbb{R}^{d})
\}$ be a stationary infinitely divisible random measure\index{infinitely divisible random measure} defined on some
probability space $(\Omega , \mathcal{A}, \mathbb{P})$ with
characteristic triplet 
$(a_{0}, 0, v_{0})$, i.e. $\Lambda $ is
purely non-Gaussian. For a known $\Lambda $-integrable function
$f:\mathbb{R}^{d} \to \mathbb{R}$ let $X = \{X(t) = \int _{\mathbb{R}
^{d}} f(t-x) \Lambda (dx), \ t \in \mathbb{R}^{d} \}$ be the infinitely
divisible moving average random field\index{random ! field} defined in
Section~\ref{subsec:reminder}.

Fix $\Delta > 0$ and suppose $X$ is observed on a regular grid
$\Delta \mathbb{Z}^{d} = \{ j\Delta , \ j \in \mathbb{Z}^{d} \}$ with
the mesh size $\Delta $, i.e. consider the random field\index{random ! field} $Y$ given by
%
\begin{equation}
\label{eq:random_field_Y}
Y = (Y_{j})_{j \in \mathbb{Z}^{d}}, \quad Y_{j} = X(j \Delta ), \ j
\in \mathbb{Z}^{d}.
\end{equation}
For a finite subset $W \subset \mathbb{Z}^{d}$ let $(Y_{j})_{j\in W}$
be a sample drawn from $Y$ and denote by $n$ the cardinality of $W$.

Throughout this paper, for any numbers $a$, $b \geq 0$, we use the
notation $a \lesssim b$ if $a \leq c b$ for some constant $c > 0$.

\begin{ass}
\label{ass:basic_assumptions}
Let the function $u:\mathbb{R}\to \mathbb{R}$ be given by $u(x) = x$.
We make the following assumptions: for some $\tau > 0$
\begin{enumerate}%
\item
$f \in L^{2+\tau }(\mathbb{R}^{d})$ has compact support;
\item
$uv_{0} \in L^{1}(\mathbb{R}) \cap L^{2}(\mathbb{R})$ is bounded;
\item
$\int _{\mathbb{R}} |x|^{1+\tau } |(uv_{0})(x)| dx < \infty $;
\item
$| \int _{\supp(f)} f(s) \mathcal{F}_{+}[uv_{0}](f(s)x) ds |
\lesssim (1+x^{2})^{-1 / 2}$ for all $x \in \mathbb{R}$;
\item
$\exists \ \varepsilon > 0$ such that the function
%
\begin{equation}
\label{eq:abs_psi_condition}
\mathbb{R}\ni x \mapsto \exp \Big (\int _{\supp(f)} \int _{0}
^{f(s)x} \textup{Im}\Big ( \mathcal{F}_{+} [uv_{0}](y) \Big ) dy ds
\Big )
\end{equation}
is contained in $H^{-1 + \varepsilon }(\mathbb{R})$.
\end{enumerate}
\end{ass}
Suppose that $\widehat{uv_{1}}$ is an estimator for $uv_{1}$ (which we
will precisely define in the next section) based on the sample $(Y_{j})_{j
\in W}$. Then, using the notation in Section~\ref{subsec:plug-in}, we
introduce the linear functional
\begin{equation*}
\hat{\mathcal{L}}_{W}: L^{2}(\mathbb{R}) \to \mathbb{R}, \quad
\hat{\mathcal{L}}_{W} v := \left \langle v, \widehat{uv_{0}} \right \rangle _{L^{2}(\mathbb{R})} = \left \langle v, \mathcal{G}_{n}^{-1}
\widehat{uv_{1}} \right \rangle _{L^{2}(\mathbb{R})}.
\end{equation*}
%
The purpose of this paper is to investigate asymptotic properties of
$\hat{\mathcal{L}}_{W} $ as the sample size $|W| = n$ tends to infinity.

\subsection{An estimator for $\mathbf{ uv_{1} }$}

In this section we introduce an estimator for the function $uv_{1}$.
Therefore, let $\psi $ denote the characteristic function\index{characteristic function} of\vadjust{\goodbreak}
$X(0)$. Then, by Assumption~\ref{ass:basic_assumptions}, (2), together
with formula~\eqref{eq:exponent_psi}, we find that $\psi $ can be
rewritten as
%
\begin{equation}
\label{eq:psi}
\psi (t) = \mathbb{E}e^{itY_{0}} = \exp \Big ( i\gamma t +
\int _{\mathbb{R}} (e^{itx} - 1) v_{1}(x) dx \Big ), \quad t \in
\mathbb{R},
\end{equation}
for some $\gamma \in \mathbb{R}$ and the L\'{e}vy density\index{L\'{e}vy ! density} $v_{1}$ given
in~\eqref{eq:levy-characterisitic-of-field}. We call $\gamma $ the
\textit{drift parameter} or shortly \textit{drift} of $X$. As a
consequence of representation~\eqref{eq:psi}, the random field $X$ is
purely non-\xch{Gaussian}{gaussian}. It is subsequently assumed that the drift
$\gamma $ is \xch{known}{knwon}.

Taking derivatives in~\eqref{eq:psi} leads to the identity
\begin{equation*}
-i \frac{\psi ^{\prime }(t)}{\psi (t)} = \gamma + \mathcal{F}_{+}[uv
_{1}](t), \quad t \in \mathbb{R}.
\end{equation*}
Neglecting $\gamma $ for the moment, this relation suggests that a
natural estimator $\widehat{\mathcal{F}_{+}[uv_{1}]}$ for $
\mathcal{F}_{+}[uv_{1}]$ is given by
\begin{equation*}
\widehat{\mathcal{F}_{+}[uv_{1}]}(t) = \frac{\hat{\theta }(t)}{
\tilde{\psi }(t)}, \quad t \in \mathbb{R},
\end{equation*}
with
\begin{equation*}
\tilde{\psi }(t) = \frac{1}{\hat{\psi }(t)} \Eins _{ \{ |
\hat{(\psi )}(t)| > n^{-1/2} \} }, \quad t \in \mathbb{R},
\end{equation*}
and $\hat{\psi }(t) = \sum _{j \in W} e^{itY_{j}}$, $\hat{\theta }(t) =
\sum _{j \in W} Y_{j} e^{itY_{j}}$ being the empirical counterparts of
$\psi $ and $\theta = -i\psi ^{\prime }$.

Now, consider for any $b > 0$ a function $K_{b}:\mathbb{R}\to
\mathbb{R}$ with the following properties:
\begin{enumerate}%
\item[\bf (K1)] $K_{b} \in L^{2}(\mathbb{R})$;
\item[\bf (K2)] $\supp(\mathcal{F}_{+} [K_{b}]) \subseteq [-b^{-1},
b^{-1}]$;
\item[\bf (K3)] $|1-\mathcal{F}_{+}[K_{b}](x)| \lesssim \min \{1, b|x|\}$
for all $x \in \mathbb{R}$.
\end{enumerate}
Then, for any $b > 0$, we define the estimator $\widehat{uv_{1}}$ for
$uv_{1}$ by
%
\begin{equation}
\label{eq:uv_1_estimator}
\widehat{uv_{1}}(t) = \mathcal{F}_{+}^{-1} \Big [
\widehat{\mathcal{F}_{+}[uv_{1}]} \mathcal{F}_{+}[K_{b}] \Big ](t) = \frac{1}{2
\pi } \int _{\mathbb{R}} e^{-itx}
\frac{\hat{\theta }(x)}{
\tilde{\psi }(x)} \mathcal{F}_{+}[K_{b}](x) dx, \quad t \in
\mathbb{R}.
\end{equation}
%
\begin{rem}
\begin{enumerate}[(a)]%
\item
If $\widehat{uv_{1}}$ is
a consistent estimator for
$uv_{1}$, it is reasonable to assume that $\gamma = 0$
(cf.~\cite{KaRoSpoWalk18}). 
Indeed, for the asymptotic
results below, the value of $\gamma $ is irrelevant. Even if
$\gamma \neq 0$, the functional $\hat{\mathcal{L}}_{W}$ estimates the
intended quantity with $\widehat{uv_{1}}$ given
in~\eqref{eq:uv_1_estimator} (cf.
Section~\ref{sec:negelcting_the_drift}).
\item
Choosing $K_{b}(x) = \frac{\sin (b^{-1}x)}{\pi x}$ yields the estimator
$\widehat{uv_{1}}$ that we introduced in~\cite{KaRoSpoWalk18}
and~\cite{GlRothSpo017}, originally designed by Comte and
Genon-Catalot~\cite{comte1} in the case when $X$ is a pure jump
L\'{e}vy process.
\end{enumerate}
\end{rem}

\subsection{Discussion and examples}

In order to explain Assumption~\ref{ass:basic_assumptions}, we prepend
the following proposition whose proof can be found in Appendix.\vadjust{\goodbreak}
%
\begin{prop}
\label{lem:properties_transferred}
Let the infinitely divisible moving average random field $X = \{X(t),\break
t \in \mathbb{R}^{d} \}$ be given as above and suppose $u(x) = x$.
\begin{enumerate}[(a)]%
\item
Let Assumption~\ref{ass:basic_assumptions}, (1) and (2) be satisfied.
Then $uv_{1} \in L^{1}(\mathbb{R}) \cap L^{2}(\mathbb{R})$ is bounded.
Moreover,
%
\begin{equation}
\label{eq:Fouriertrafo_uv_1}
\mathcal{F}_{+}[uv_{1}](x) = \int _{\supp(f)} f(s) \mathcal{F}
_{+}[uv_{0}](f(s)x) ds, \quad \text{for all} \ x \in \mathbb{R},
\end{equation}
that is, the expression in Assumption~\ref{ass:basic_assumptions}, (4) is valid.
\item
Let Assumption~\ref{ass:basic_assumptions}, (1) and (3) hold true. Then
$\int _{\mathbb{R}} |x|^{2+\tau } |(uv_{1})(x)|dx < \infty $ (also in
the case when $\tau = 0$).
\item
Assumption~\ref{ass:basic_assumptions}, (5) is satisfied if and only
if the function $\mathbb{R}\ni x \mapsto (1+x^{2})^{-\frac{1}{2} +
\varepsilon } \frac{1}{\psi (x)}$, with $\psi $ given
in~\eqref{eq:psi}, for some
$\varepsilon > 0$ belongs to $L^{2}(\mathbb{R})$.
\end{enumerate}
\end{prop}
The compact support property in Assumption~\ref{ass:basic_assumptions},
(1) ensures that the random field $(Y_{j})_{j \in \mathbb{Z}^{d}}$ is
$m$-dependent with $m > \Delta ^{-1} \textup{diam}(\supp(f))$
(cf. Lemma~\ref{lem:m_depend_X}). In particular, $m$ increases when the
grid size $\Delta $ of the sample is decreasing. Moreover, compact
support of $f$ together with $f \in L^{2+\tau }(\mathbb{R})$ implies
that $f \in L^{q}(\mathbb{R})$ for all $0 < q \leq 2+\tau $.
Consequently, $f$ fulfills the integrability
condition~\eqref{eq:integrability_condition_f}. In contrast, if $f$ does
not have compact support, the $\Lambda $-integrability only ensures
$f \in L^{2}(\mathbb{R})$.

Assumption~\ref{ass:basic_assumptions}, (3) is a moment assumption on
$\Lambda $. More precisely, it is satisfied if and only if
\begin{equation*}
\mathbb{E}|\Lambda (A)|^{2+\tau } < \infty
\end{equation*}
for all $A \in \mathcal{E}_{0}(\mathbb{R}^{d})$,
cf.~\cite{sato1999}. By
Proposition~\ref{lem:properties_transferred}, (b), this assumption also
implies $\mathbb{E}|X(0)|^{2+\tau } < \infty $ in our setting.

As a consequence of Proposition~\ref{lem:properties_transferred}, (a)
and (c), Assumption~\ref{ass:basic_assumptions}, (4) ensures regularity
of $uv_{1}$ whereas (5) yields the polynomial decay of $\psi $. It was shown
in~\cite[Theorem 3.10]{GlRothSpo017} that $\psi $ and $uv_{1}$ are
connected via the relation
\begin{equation*}
|\psi (x)| = \exp \Big ( - \int _{0}^{x} \textup{Im}\big ( \mathcal{F}
_{+}[uv_{1}](y) \big ) dy \Big ), \quad x \in \mathbb{R};
\end{equation*}
hence, more regularity of $uv_{1}$ ensures slower decay rates for
$|\psi (x)|$ as $x \to \pm \infty $. Further results on the polynomial decay
of infinitely divisible characteristic functions\index{infinitely divisible characteristic functions} as well as sufficient
conditions for this property to hold can be found
in~\cite{Trabs14}.

Let us give some examples of $\Lambda $ and $f$ \xch{satisfying}{satisfiying}
Assumption~\ref{ass:basic_assumptions}, (1)--(5).

\begin{expl}[Gamma random measure]
\label{expl:gamma_random_measure}
Fix $b > 0$ and let for any $x \in \mathbb{R}$, $v_{0}(x) = x^{-1} e
^{-bx} \Eins _{(0,\infty )}(x)$. Clearly,
Assumption~\ref{ass:basic_assumptions}, (2) and (3) are 
satisfied for any
$\tau > 0$. The Fourier transform of $uv_{0}$ is given by $
\mathcal{F}_{+}[uv_{0}](x) = (b-ix)^{-1}$, $x \in \mathbb{R}$; hence
\begin{equation*}
\int _{\supp(f)} f(s) \mathcal{F}_{+}[uv_{0}] (f(s) x) ds =
\int _{\supp(f)} \frac{f(s)}{b-if(s)x} ds, \quad x \in
\mathbb{R}.
\end{equation*}
The latter identity shows that Assumption~\ref{ass:basic_assumptions},
(4) holds true for any integrable $f$ with a compact support. Moreover,
a simple calculation yields that for any $x \in \mathbb{R}$,
Assumption~\ref{ass:basic_assumptions}, (5) becomes
%
\begin{equation}
\label{eq:five_equivalent}
\int _{\mathbb{R}} (1+x^{2})^{-1+\varepsilon } \exp \Big (
\int _{\supp(f)} \log \big ( 1 + \frac{x^{2} f^{2}(s)}{b}
\big ) ds \Big ) dx < \infty .
\end{equation}
This condition is fulfilled for any $\varepsilon < \frac{1}{2} -
\alpha $ if
\begin{equation*}
\alpha := \int _{\supp(f)} \max \Big \{ 1, \frac{f^{2}(s)}{b}
\Big \} ds < \frac{1}{2}.
\end{equation*}
\end{expl}

\subsection{Consistency of $\hat{\mathcal{L}}_{W}$}

In this section, we give an upper bound for the estimation error
$\mathbb{E}|\hat{\mathcal{L}}_{W} v - \mathcal{L}v|$ that allows to
derive conditions under which $\hat{\mathcal{L}}_{W}$ is consistent for
the linear functional $\mathcal{L}: L^{2}(\mathbb{R}) \to \mathbb{R}$
given by
\begin{equation*}
\mathcal{L}v = \left \langle v, uv_{0} \right \rangle , \quad v
\in L^{2}(\mathbb{R}).
\end{equation*}
With the notations from Section~\ref{subsec:plug-in}, we have that the
adjoint operator $\mathcal{G}^{-1 \ast } :\break \textup{Image}(\mathcal{G})
\to L^{2}(\mathbb{R})$ of $\mathcal{G}^{-1}$ is given by
%
\begin{equation}
\label{eq:Gi_inverse_adjoint}
\mathcal{G}^{-1 \ast } v = \mathcal{M}^{-1} \mathcal{F}_{\times }^{-1}
\Big ( \frac{1}{\bar{\mu }_{f}} \mathcal{F}_{\times }\mathcal{M}v
\Big ), \quad v \in \textup{Image}(\mathcal{G}),
\end{equation}
where $\bar{\mu }_{f}$ denotes the complex conjugate function of
$\mu _{f}$. Moreover, the adjoint $\mathcal{G}_{n}^{-1 \ast }:L^{2}(
\mathbb{R}) \to L^{2}(\mathbb{R})$ of $\mathcal{G}_{n}^{-1}$ writes as
\begin{equation*}
\mathcal{G}_{n}^{-1 \ast } v = \mathcal{M}^{-1} \mathcal{F}_{\times }
^{-1} \Big ( \frac{1}{\bar{\mu }_{f,n}} \mathcal{F}_{\times }
\mathcal{M}v \Big ), \quad v \in L^{2}(\mathbb{R}),
\end{equation*}
with $\frac{1}{\bar{\mu }_{f,n}} = \frac{1}{\bar{\mu }_{f}}
\Eins _{ \{ |\bar{\mu }_{f}| > a_{n} \} }$. Notice that
$\mathcal{G}_{n}^{-1 \ast }$ is a bounded operator whereas $
\mathcal{G}^{-1 \ast }$ is unbounded in general.

\begin{rem}
Notice that $\mathcal{G}_{n}^{-1 \ast } = \mathcal{G}^{-1 \ast }$ if
$a_{n} = 0$ for any $n \in \mathbb{N}$. Hence, $\mathcal{G}_{n}^{-1
\ast }\widehat{uv_{1}}$ in this case only is well-defined if
$\widehat{uv_{1}}$ is 
an element of $\textup{Image}(
\mathcal{G}^{\ast })$ 
what is indeed a very restrictive assumption. For
a detailed discussion we refer to~\cite{GlRothSpo017}.
\end{rem}
With the previous notations we now derive an upper bound for
$\mathbb{E}|\hat{\mathcal{L}}_{W} v - \mathcal{L}v|$. Therefore, recall
condition ($\mathbf{U}_{\beta}$) from Section~\ref{subsec:plug-in}.

\begin{lem}
\label{theo:upper_bound_estimation_error}
Let $\gamma = 0$ and suppose Assumption~\ref{ass:basic_assumptions}, (1)--(3) hold true for some $\tau \geq 0$. Moreover, let condition
$(\mathbf{U}_{\beta })$ be satisfied for some $\beta \geq 0$ and assume that
$K_{b}:\mathbb{R}\to \mathbb{R}$ is a function with properties
(\textbf{K1})--(\textbf{K3}). Then
%
\begin{align}
\label{eq:upper_bound_uv_1}
\mathbb{E}|\hat{\mathcal{L}}_{W} v - \mathcal{L}v| \leq{}
& \frac{S}{\sqrt{
\pi }} \mathbb{E}|Y_{0}| \Big ( \frac{n}{b} \Big )^{1/2}
\|\big (
\mathcal{G}_{n}^{-1 \ast } - \mathcal{G}^{-1 \ast } \big ) v\|_{L^{2}(
\mathbb{R})}\nonumber
\\
& + \frac{1}{2\pi } \left \langle |\mathcal{F}_{+}[\mathcal{G}^{-1
\ast } v] |, |\mathcal{F}_{+}[uv_{1}]| |1-\mathcal{F}_{+}[K_{b}]| \right \rangle _{L^{2}(\mathbb{R})}\nonumber
\\
& + \frac{c \cdot S}{2 \pi \sqrt{n}} \Big ( \sqrt{\mathbb{E}|Y
_{0}|^{2}} + \|uv_{1}\|_{L^{1}(\mathbb{R})} \Big ) \int _{\mathbb{R}} \frac{|
\mathcal{F}_{+}[\mathcal{G}^{-1 \ast }v]|(x)}{|\psi (x)|} dx
\end{align}
for any $v \in \textup{Image}(\mathcal{G})$ such that $\int _{
\mathbb{R}} \frac{|\mathcal{F}_{+}[\mathcal{G}^{-1 \ast }v](x)|}{|
\psi (x)|} dx < \infty $, where  $c > 0$ is some constant and
$S:=\sup _{b > 0, \ x \in \mathbb{R}} |\mathcal{F}_{+}[K_{b}](x)|$\xch{.}{.-}
\end{lem}

A proof of Lemma~\ref{theo:upper_bound_estimation_error} as well as of
Theorem~\ref{cor:order_of_convergence} below can be found in Appendix.
%
\begin{theo}
\label{cor:order_of_convergence}
Fix $\gamma \in \mathbb{R}$. Suppose that condition $(\mathbf{U}_{
\beta _{1}})$ is satisfied for some $\beta _{1} \geq 0$ and let
$v \in L^{2}(\mathbb{R})$ be such that $\mathcal{G}^{-1 \ast }v \in H
^{1}(\mathbb{R})$, $\frac{\mathcal{F}_{+}[\mathcal{G}^{-1 \ast }v]}{
\psi } \in L^{1}(\mathbb{R})$ and
%
\begin{equation}
\label{eq:contained_in_range_condition}
\begin{split}
(\mathcal{M}v)(\exp (\, \cdot \,)), \ (\mathcal{M}v)(-\exp (\, \cdot
\,)) \in H^{\beta _{2}}(\mathbb{R})
\end{split}
\end{equation}
for some $\beta _{2} > \beta _{1}$. Moreover, let $a = a_{n}$ and
$b = b_{n}$ be sequences with the properties
\begin{equation*}
a_{n} \to 0, \quad b_{n} \to 0 \quad \text{and} \quad a_{n} = o
\Big ( \Big ( \frac{n}{b_{n}} \Big )^{\frac{\beta _{1}}{2(\beta _{1} -
\beta _{2})}} \Big ), \quad  \text{as} \ n \to \infty,
\end{equation*}
and assume that conditions (\textbf{K1})--(\textbf{K3}) are fulfilled.
Then, under  Assumption~\ref{ass:basic_assumptions}, \mbox{(1)--(4)},
$\mathbb{E}|\hat{\mathcal{L}}_{W} v - \mathcal{L}v| \to 0$ as
$n \to \infty $ with the order of convergence given by
\begin{equation*}
\mathbb{E}|\hat{\mathcal{L}}_{W} v - \mathcal{L}v| = \mathcal{O}
\Big ( a_{n}^{\frac{\beta _{2}}{\beta _{1}} - 1 } \sqrt{ \frac{n}{b
_{n}} } + b_{n} + \frac{1}{\sqrt{n}} \Big ).
\end{equation*}
\end{theo}

\begin{rem}
\begin{enumerate}[(a)]%
\item
Notice that condition $(\mathbf{U}_{\beta })$ ensures uniqueness of
$uv_{0} \in L^{2}(\mathbb{R})$ as a solution of $\mathcal{G}uv_{0} = uv
_{1}$. In Lemma~\ref{theo:upper_bound_estimation_error}, it can be
replaced by the more (and most) general assumption $m_{f,\pm } \neq 0$
almost everywhere on $\mathbb{R}$. Moreover, condition \textbf{(K3)} can
be replaced by $\sup _{b > 0, \ x \in \mathbb{R}} |\mathcal{F}_{+}[K
_{b}](x)| < \infty $ in Lemma~\ref{theo:upper_bound_estimation_error}.
\item
In order to deduce the convergence rate in
Theorem~\ref{cor:order_of_convergence} explicitely,
condition~\eqref{eq:contained_in_range_condition} is essential.
Moreover, it ensures that the function $v$
belongs to the range of
$\mathcal{G}$ (cf.~\cite[Theorem 3.1]{GlRothSpo017}); hence the
expression $\mathcal{G}^{-1 \ast } v$ is well-defined.
\item
The condition $\mathcal{G}^{-1 \ast }v \in H^{1}(\mathbb{R})$ in
Theorem~\ref{cor:order_of_convergence} can be dropped if $\gamma = 0$.
\item
Under the conditions of Theorem~\ref{cor:order_of_convergence}, the
convergence rate of $\mathbb{E}|\hat{\mathcal{L}}_{W} v - \mathcal{L}v|
\to 0$ is at least $\mathcal{O}(n^{-1/2})$ as $n \to \infty $, provided
that
\begin{equation*}
a_{n} = o \Big ( \Big ( \frac{n}{\sqrt{b_{n}}} \Big )^{\frac{\beta _{1}}{
\beta _{1} - \beta _{2}}} \Big )
\quad \text{and} \quad b_{n} =
\mathcal{O}\Big ( \frac{1}{\sqrt{n}} \Big ),\quad  \text{as} \ n \to \infty
.
\end{equation*}
\end{enumerate}
\end{rem}

We close this section with the following example, showing that the
functions $g_{t}$ considered in~\cite[p. 3309]{nickl} may 
belong to the range of $\mathcal{G}^{-1 \ast }$.

\begin{expl}
Fix $t > 0$ and let $v(x) = \frac{1}{x} \Eins _{\mathbb{R}\backslash
[-t,t]}(x)$, $x \in \mathbb{R}$. Apparently, $v \in L^{2}(\mathbb{R})$
fulfills condition~\eqref{eq:contained_in_range_condition} for any
$\beta _{2} > 0$. Let for some fixed $\lambda , \ \theta > 0$,
$f(s) = e^{-\lambda s}\Eins _{(0,\theta )}(s)$, $s \in
\mathbb{R}$. Then a simple computation yields that $(\mathbf{U}_{
\beta _{1}})$ is satisfied with $\beta _{1} = 1$. Moreover, for all
$x \neq 0$
\begin{equation*}
(\mathcal{G}^{-1 \ast }v)(x) = \frac{1}{2 x} \log \Big ( \frac{|x|}{t}
\Big ) \Eins _{(t,t e^{\lambda \theta }]}(|x|)
+ \frac{\lambda
\theta }{2 x}\Eins _{(t e^{\lambda \theta }, \infty )}(|x|);
\end{equation*}
hence, $\mathcal{G}^{-1 \ast }v \in H^{1}(\mathbb{R})$. Since
\begin{equation*}
\Big \| \frac{\mathcal{F}_{+}[\mathcal{G}^{-1 \ast }v]}{\psi } \Big
\|_{L^{1}(\mathbb{R})} \leq \| \mathcal{G}^{-1 \ast } v \|_{H^{1}(
\mathbb{R})} \Big \|
\frac{(1+\, \cdot \,^{2})^{-\frac{1+\varepsilon
}{2}}}{\psi } \Big \|_{L^{2}(\mathbb{R})},
\end{equation*}
any random measure\index{random ! measure} $\Lambda $ satisfying
Assumption~\ref{ass:basic_assumptions}, (5) yields $\frac{
\mathcal{F}_{+}[\mathcal{G}^{-1 \ast }v]}{\psi } \in L^{1}(\mathbb{R})$
(cf. Proposition~\ref{lem:properties_transferred}, (c)).
\end{expl}

\subsection{A central limit theorem for $\hat{\mathcal{L}}_{W}$}%
\label{sec:clts}
Provided the assumptions of Theorem~\ref{cor:order_of_convergence} are
satisfied,
\begin{equation*}
\err_{W} (v) := \sqrt{n}\ (\hat{\mathcal{L}}_{W} v - \mathcal{L}v)
\end{equation*}
is bounded in mean. In this section, we give conditions under which
$\err_{W}(v)$ is asymptotically Gaussian. For this purpose, introduce
the following notation.

\begin{definition}
\label{def:admissible_function}
Let Assumption~\ref{ass:basic_assumptions} be satisfied and suppose that
condition ($\mathbf{U}_{\beta _{1}})$ is fulfilled for some
$\beta _{1} > 0$. A function $v \in L^{2}(\mathbb{R})$ is called
\textit{admissible} of index $(\xi , \beta _{2})$ if
\begin{enumerate}[(i)]%
\item
$\mathcal{G}^{-1 \ast } v \in H^{\frac{3}{2} - \varepsilon }(
\mathbb{R})$,
\item
$(\mathcal{M}v)(\exp (\, \cdot \,))$, $(\mathcal{M}v)(-\exp (\,
\cdot \,)) \in H^{\beta _{2}}(\mathbb{R})$ for some $\beta _{2} > \beta
_{1}$ and
\item
$|\mathcal{F}_{+}[\mathcal{G}^{-1 \ast }v](x)| \lesssim (1+x^{2})^{-
\xi /2}$ for some $\xi > 2(1-\varepsilon ) - \Big ( \frac{1}{2} -
\varepsilon \Big ) \frac{1+\tau }{2+\tau }$.
\end{enumerate}
The linear subspace of all admissible functions of index $(\xi , \beta
_{2})$ is denoted by $\mathcal{U}(\xi , \beta _{2})$.
\end{definition}

\begin{rem}
\label{rem:admissible_function_discussed}
\begin{enumerate}[(a)]%
\item
The parameters $\varepsilon $ and $\tau $ describe the size of
$\mathcal{U}(\xi , \beta _{2})$. In particular, for larger values of
$\varepsilon $ and $\tau $, the set of admissible functions\index{admissible functions} is
increasing and vice versa.
\item
Assumption~\ref{ass:basic_assumptions}, (5) implies $\varepsilon <
\frac{1}{2}$ (otherwise $\frac{1}{|\psi (x)|} \to 0$ as $|x| \to
\infty $); hence Definition~\ref{def:admissible_function}, (i) yields
that $\mathcal{F}_{+}[\mathcal{G}^{-1 \ast } v] \in L^{1}(\mathbb{R})$.
\item
Clearly, the lower bound for $\xi $ in
\xch{Definition}{Defintion}~\ref{def:admissible_function}, (iii) can be replaced by
$\xi > \frac{7}{4} - \frac{3}{2}\varepsilon $. Nevertheless, since
our purpose is to point out the influence of $\tau $ on the set of
admissible functions,\index{admissible functions} we do not use this simplification.
\item
It immediately follows from formula~\eqref{eq:Gi_inverse_adjoint} that
$\mathcal{G}^{-1 \ast } v \in H^{\delta }(\mathbb{R})$ if and only if
$\mathcal{G}^{-1} v \in H^{\delta }(\mathbb{R})$.
\end{enumerate}
\end{rem}

For any $j \in W$ and any admissible function $v \in \mathcal{U}(
\xi , \beta _{2})$, introduce the random variables
\begin{equation*}
\begin{split}
Z_{j,v}^{(1)}
& = \frac{1}{2\pi } Y_{j} \mathcal{F}_{+} \Big [ \frac{
\mathcal{F}_{+}[\mathcal{G}^{-1 \ast }v](- \, \cdot \,)}{\psi (\,
\cdot \,)} \Big ](Y_{j}) \quad \text{and}
\\
Z_{j,v}^{(2)}
& = \frac{i}{2\pi } \mathcal{F}_{+} \Big [ \mathcal{F}
_{+}[\mathcal{G}^{-1 \ast }v](- \, \cdot \,) \Big ( \frac{1}{\psi }
\Big )^{\prime }\Big ](Y_{j}).
\end{split}
\end{equation*}
In the sequel, it is assumed that the random field\index{random ! field} $Y$ introduced
in~\eqref{eq:random_field_Y} is observed on a sequence $(W_{k})_{k
\in \mathbb{N}}$ of regularly growing\index{regularly growing} observation windows (cf.
Section~\ref{subsec:regularly_growing_sets}). To avoid longer notations,
we drop the index $k$ in this notation and shortly write $W$ instead of
$W_{k}$. Moreover, we denote by $n$ $(= n(k))$ the cardinality of
$W$.

With the previous notation, we now can formulate the main result of this
section.
%
\begin{theo}
\label{theo:clt_univariate}
Fix $m \in \mathbb{N}$, $m > \Delta ^{-1} \textup{diam}( \supp(f)
)$. Let Assumption~\ref{ass:basic_assumptions} be satisfied and suppose
that conditions \textbf{(K1)}--\textbf{(K3)}\vadjust{\goodbreak} are fulfilled. Moreover,
let for some $\eta > 0$ the sequences $a_{n}$ and $b_{n}$ be given by
\begin{equation*}
a_{n} = o \Big ( \Big ( \frac{n}{\sqrt{b_{n}}} \Big )^{\frac{\beta _{1}}{
\beta _{1} - \beta _{2}}} \Big ) \quad \text{and} \quad b_{n} \approx n
^{-\frac{1}{1-2\varepsilon }} (\log n )^{\eta + \frac{1}{1-2\varepsilon
}}, \quad  \text{as} \ n \to \infty .
\end{equation*}
Then, as $W$ is regularly growing\index{regularly growing} to infinity,
\begin{equation*}
\err_{W}(v) \overset{d}{\to } N_{v},
\end{equation*}
for any admissible function $v \in \mathcal{U}(\xi , \beta _{2})$, where
$N_{v}$ is a Gaussian random variable\index{Gaussian random variable} with zero expectation and variance
given by
%
\begin{equation}
\label{eq:asymptotic_variance}
\sigma ^{2}_{v} = \sum _{j \in \mathbb{Z}^{d}: \ \|j\|_{\infty }\leq m}
\mathbb{E}\Big [ \Big ( Z_{j,v}^{(1)} - Z_{j,v}^{(2)} \Big ) \Big ( Z_{0,v}
^{(1)} - Z_{0,v}^{(2)} \Big ) \Big ].
\end{equation}
\end{theo}
A proof of Theorem~\ref{theo:clt_univariate} can be found in
Section~\ref{proof:clt_univariate}.

\begin{rem}
Unfortunately, we could not provide a rate for the convergence\break
$\err_{W}(v) \overset{d}{\to } N_{v}$ in
Theorem~\ref{theo:clt_univariate}. Therefore, it would be sufficient to
provide, e.g., $L^{1}(\Omega , \mathbb{P})$-rates for the convergence
$\sup _{x} |\hat{\psi }(x) - \psi (x)|$, $\sup _{x} |\hat{\theta }(x) -
\theta (x)| \to 0$ (as $|W| \to \infty $), that seems to be a hard
problem in the dependent observations setting.
\end{rem}

\begin{cor}
\label{cor:clt_multivariate}
Let the assumptions of Theorem~\ref{theo:clt_univariate} hold true.
Then, as $W$ is regularly growing\index{regularly growing} to infinity,
\begin{equation*}
(\err_{W}(v_{1}), \dots , \err_{W}(v_{k}))^{\top }\overset{d}{\to } N
_{v_{1},\dots ,v_{k}},
\end{equation*}
for any $v_{1} \in \mathcal{U}(\xi _{1}, \beta _{2}^{(1)}), \dots , v
_{k} \in \mathcal{U}(\xi _{k}, \beta _{2}^{(k)})$, where $N_{v_{1},
\dots ,v_{k}}$ is a centered Gaussian random vector with covariance
matrix $(\Sigma _{s,t})_{s,t=1,\dots ,k}$ given by
\begin{equation*}
\Sigma _{s,t} = \sum _{j \in \mathbb{Z}^{d}: \ \|j\|_{\infty }\leq m}
\mathbb{E}\Big [ \Big ( Z_{j,v_{s}}^{(1)} - Z_{j,v_{s}}^{(2)} \Big )
\Big ( Z_{0,v_{t}}^{(1)} - Z_{0,v_{t}}^{(2)} \Big ) \Big ].
\end{equation*}
\end{cor}

\begin{proof}
Suppose $v_{1} \in \mathcal{U}(\xi _{1}, \beta _{2}^{(1)}),\dots ,v_{k}
\in \mathcal{U}(\xi _{k}, \beta _{2}^{(k)})$ and, for arbitrary numbers
$\lambda _{1},\dots ,\lambda _{k} \in \mathbb{R}$, let $v = \sum _{l=1}
^{k} \lambda _{l} v_{l}$. Then a simple calculation yields
\begin{equation*}
\sum _{l=1}^{k} \lambda _{l} \err_{W}(v_{l}) = \sqrt{n} \ (
\hat{\mathcal{L}}_{W} v - \mathcal{L}v).
\end{equation*}
Since $v \in \mathcal{U}(\min _{l}\xi _{l}, \min _{l} \beta _{2}^{(l)})$,
by Theorem~\ref{theo:clt_univariate}, $\sqrt{n} \ (
\hat{\mathcal{L}}_{W} v - \mathcal{L}v) \overset{d}{\to } N_{v}$, where
$N_{v}$ is a Gaussian random variable\index{Gaussian random variable} with zero expectation and variance
given in~\eqref{eq:asymptotic_variance}. Now, let $(T_{1}, \dots , T
_{k})^{\top }$ be a zero mean Gaussian random vector with covariance
given by $(\Sigma _{s,t})_{s,t=1,\dots ,k}$. Using linearity of
$\mathcal{F}_{+}$ and $\mathcal{G}^{-1 \ast }$, a short computation
shows that
\begin{equation*}
N_{v} \overset{d}{=} \sum _{l=1}^{k} \lambda _{l} T_{l};
\end{equation*}
hence, the assertion follows by the Cram\'{e}r--Wold theorem
(cf.~\cite{billingsley2012}).
\end{proof}

\section{Proof of Theorem~\ref{theo:clt_univariate}}%
\label{proof:clt_univariate}
In order to prove Theorem~\ref{proof:clt_univariate}, we adopt the
strategy of the proof of~\cite[Theorem 2]{nickl}. Nevertheless, the
main difficulty in our setting is that the observations $(Y_{j})_{j
\in W}$ are not independent; hence the classical theory cannot be
applied here. Instead, we use asymptotic results for partial sums of
$m$-dependent random fields (cf.~\cite{Chen2004}) in combination
with the theory developed by Bulinski and Shashkin
in~\cite{Bulinski07} for weakly dependent random fields.\index{random ! field}

We start with the following lemma.

\begin{lem}
\label{lem:important_lemma}
Let $\gamma = 0$ and suppose that $v \in \mathcal{U}(\xi , \beta _{2})$
is an admissible function. Then Assumption~\ref{ass:basic_assumptions}
implies:
\begin{enumerate}%
\item
$xP$ has a bounded Lebesgue density on $\mathbb{R}$, where $P$ denotes the
distribution of $X(0)$.
\item
$\Big ( \frac{1}{\psi } \Big )^{\prime }\in L^{2}(\mathbb{R}) \cap L
^{\infty }(\mathbb{R})$ and $\frac{1}{|\psi (x)|} \lesssim (1+|x|)^{
\frac{1}{2} - \varepsilon }$ for all $x \in \mathbb{R}$.
\item
$\mathcal{F}_{+}[\mathcal{G}^{-1 \ast }v]$, $\frac{\mathcal{F}_{+}[
\mathcal{G}^{-1 \ast }v]}{\psi } \in L^{1}(\mathbb{R}) \cap L^{2}(
\mathbb{R})$.
\end{enumerate}
\end{lem}

\begin{proof}
\begin{enumerate}%
\item
Let $\mu (dx) = (uv_{1})(x) dx$. By
Proposition~\ref{lem:properties_transferred}, (a), $uv_{1} \in L^{1}(
\mathbb{R})$; hence, $\mu $ defines a finite signed measure on
$\mathbb{R}$. Since $\theta = \psi \mathcal{F}_{+}[uv_{1}]$, we conclude
that
\begin{equation*}
\mathcal{F}_{+}[xP](t) = \theta (t) = \mathcal{F}_{+}[P](t)
\mathcal{F}_{+}[uv_{1}](t) = \mathcal{F}_{+}[\mu \ast P](t),
\end{equation*}
i.e. $xP(dx) = (\mu \ast P)(dx)$; thus, $xP$ has the density given by
$\frac{d[xP]}{dx} =\break  \int _{\mathbb{R}} (uv_{1})(x-y) P(dy)$ and
consequently $\|\frac{d[xP]}{dx}\|_{L^{\infty }(\mathbb{R})} \leq \|uv
_{1}\|_{L^{\infty }(\mathbb{R})}$.
\item
By Assumption~\ref{ass:basic_assumptions}, (4), (5),
Proposition~\ref{lem:properties_transferred}, (a), (c) and the
\xch{Cauchy--Schwarz}{Cauchy-Schwart} inequality, we obtain for any $x \in \mathbb{R}$,
\begin{equation*}
\begin{split}
\frac{1}{|\psi (x)|} ={}
& 1 + \int _{0}^{x} \Big ( \frac{1}{\psi }
\Big )^{\prime }(t) dt
= 1 + \int _{0}^{x}
\frac{|\theta (t)|}{|\psi (t)|^{2}} dt
\\
={}
& 1 + \int _{0}^{x} \frac{|\mathcal{F}_{+}[uv_{1}](t)|}{|\psi (t)|}
dt
\\
\lesssim{}
& 1 + \int _{0}^{x} (1+t^{2})^{-\frac{\varepsilon }{2}} \frac{(1+t
^{2})^{-\frac{1-\varepsilon }{2}}}{|\psi (t)|} dt
\\
\leq{}
& 1 + \Big \| \frac{(1+\, \cdot \,^{2})^{-
\frac{1-\varepsilon }{2}}}{\psi } \Big \|_{L^{2}(\mathbb{R})}
\Big (
\int _{0}^{x} (1+t^{2})^{-\varepsilon } dt \Big )^{1/2}
\\
\lesssim{}
& (1+|x|)^{\frac{1}{2}-\varepsilon }.
\end{split}
\end{equation*}
Further, we have for any $x \in \mathbb{R}$,
\begin{equation*}
\left | \Big ( \frac{1}{\psi } \Big )^{\prime }(x) \right |
= \frac{|
\mathcal{F}_{+}[uv_{1}](x)|}{|\psi (x)|} \lesssim (1+|x|)^{-
\frac{1}{2}-\varepsilon }.
\end{equation*}
The last expression is bounded and square integrable, hence
$\Big ( \frac{1}{\psi } \Big )^{\prime }\in L^{2}(\mathbb{R}) \cap L
^{\infty }(\mathbb{R})$.
\item
$\mathcal{F}_{+}[\mathcal{G}^{-1 \ast }v] \in L^{1}(\mathbb{R})
\cap L^{2}(\mathbb{R})$ immediately follows from
Definition~\ref{def:admissible_function}, (i) (cf.
Remark~\ref{rem:admissible_function_discussed}, (b)). Moreover, by
Proposition~\ref{lem:properties_transferred}, (a), we find that
\begin{equation*}
\begin{split}
\int _{\mathbb{R}}\frac{|\mathcal{F}_{+}[\mathcal{G}^{-1 \ast }v](x)|}{|
\psi (x)|}dx \leq \|\mathcal{G}^{-1 \ast }v\|_{H^{1-\varepsilon }(
\mathbb{R})}
\Big \| \frac{(1+\, \cdot \,^{2})^{-\frac{1-\varepsilon
}{2}}}{\psi } \Big \|_{L^{2}(\mathbb{R})},
\end{split}
\end{equation*}
where the latter is finite due to
Definition~\ref{def:admissible_function}, (i). The bound in part (2)
finally yields
\begin{equation*}
\begin{split}
\int _{\mathbb{R}}\frac{|\mathcal{F}_{+}[\mathcal{G}^{-1 \ast }v](x)|^{2}}{|
\psi (x)|^{2}}dx \lesssim{}
& \int _{\mathbb{R}}|\mathcal{F}_{+}[
\mathcal{G}^{-1 \ast }v](x)|^{2} (1+|x|^{2})^{\frac{1}{2}-\varepsilon
} dx
\\
={}
& \|\mathcal{G}^{-1 \ast }v\|_{H^{1-2\varepsilon }(\mathbb{R})} <
\infty .\qedhere
\end{split}
\end{equation*}
\end{enumerate}
\end{proof}

In order to prove Theorem~\ref{theo:clt_univariate}, consider the
following decomposition that can be obtained by the isometry property\index{isometry property} of
$\mathcal{F}_{+}$:
\begin{equation*}
\err_{W}(v) = \sqrt{n} \big ( \hat{\mathcal{L}}_{W} v - \mathcal{L}v
\big ) = \frac{1}{2\pi } \Big [ E_{1} + E_{2} + E_{3} + E_{4} \Big ] + E
_{5},
\end{equation*}
with $E_{1},\dots ,E_{5}$ given by
\begin{equation*}
\begin{split}
E_{1} ={}
& \sqrt{n} \Big < \mathcal{F}_{+}[\mathcal{G}^{-1 \ast }v],
\Big \{ \frac{\hat{\theta } - \theta }{\psi } - i \Big ( \frac{1}{
\psi } \Big )^{\prime }(\hat{\psi } - \psi ) \Big \} \mathcal{F}_{+}[K
_{b}] \Big >_{L^{2}(\mathbb{R})}
\\
E_{2} ={}
& \sqrt{n} \Big < \mathcal{F}_{+}[\mathcal{G}^{-1 \ast }v],
\Big \{ R_{n} + \theta \frac{\psi - \hat{\psi }}{\psi ^{2}} \Eins
_{ \{ |\hat{\psi }| \leq n^{-1/2} \} } \Big \} \mathcal{F}_{+}[K_{b}]
\Big >_{L^{2}(\mathbb{R})}
\\
E_{3} ={}
& \sqrt{n} \Big < \mathcal{F}_{+}[\mathcal{G}^{-1 \ast }v],
\frac{\theta }{\psi } (\mathcal{F}_{+}[K_{b}]-1) \Big >_{L^{2}(
\mathbb{R})}
\\
E_{4} ={}
& \sqrt{n} \Big < \mathcal{F}_{+}[\mathcal{G}^{-1 \ast }v],
\mathcal{F}_{+}[K_{b}] \Big \{ \theta \frac{\hat{\psi } - \psi }{\psi
^{2}} - \frac{\hat{\theta }}{\psi } \Big \} \Eins_{ \{ |
\hat{\psi }| \leq n^{-1/2} \} } \Big >_{L^{2}(\mathbb{R})}
\\
E_{5} ={}
& \sqrt{n} \Big < (\mathcal{G}_{n}^{-1 \ast }-\mathcal{G}
^{-1 \ast })v, \widehat{uv_{1}} \xch{\Big >_{L^{2}(\mathbb{R})},}{\Big >_{L^{2}(\mathbb{R})}.}
\end{split}
\end{equation*}
and $R_{n} =  \bigl( 1-\frac{\hat{\psi }}{\psi }  \bigr)  \bigl( \frac{
\hat{\theta }}{\tilde{\psi }} - \frac{\theta }{\psi }  \bigr)$. We call
the expression $E_{1}$ \textit{main stochastic term\index{main stochastic term}} and the expression
$E_{2}$ \textit{remainder term}\index{remainder term}.

Subsequently, we give a step by step proof for
Theorem~\ref{theo:clt_univariate} by considering each of the above terms
$E_{1},\dots ,E_{5}$ seperately.

We first show that the deterministic term $E_{3}$ tends to zero as the
sample size $n$ tends to infinity.
%
\begin{lem}
Suppose $\gamma = 0$. Then, under the conditions of
Theorem~\ref{theo:clt_univariate},
\begin{equation*}
E_{3} = \sqrt{n} \Big < \mathcal{F}_{+} [\mathcal{G}^{-1 \ast }v], \frac{
\theta }{\psi } (\mathcal{F}_{+}[K_{b}] - 1) \Big >_{L^{2}(\mathbb{R}
^{\times })}
\to 0, \quad \text{as} \ n \to \xch{\infty ,}{\infty .}
\end{equation*}
for any admissible function $v \in \mathcal{U}(\xi , \beta _{2})$.
\end{lem}
\begin{proof}
Taking into account that $\big | \frac{\theta }{\psi } \big | = |
\mathcal{F}_{+}[uv_{1}]|$, Assumption~\ref{ass:basic_assumptions}, (4),
together with Proposition~\ref{lem:properties_transferred}, (a) and
condition \textbf{(K3)} yield
\begin{equation*}
\begin{split}
|E_{3}| \leq{}
& \sqrt{n} \int _{\mathbb{R}} |\mathcal{F}_{+} [
\mathcal{G}^{-1 \ast }v](x)| |\mathcal{F}_{+}[uv_{1}](x)| |1-
\mathcal{F}[K_{b_{n}}](x)| dx
\\
\lesssim{}
& b_{n} \sqrt{n} \int _{\mathbb{R}} |\mathcal{F}_{+} [
\mathcal{G}^{-1 \ast }v](x)| dx,
\end{split}
\end{equation*}
where the last line is finite due to Lemma~\ref{lem:important_lemma}.
Moreover, since, $b_{n} = o(n^{-1/2})$ it tends to $0$ as $n \to
\infty $.
\end{proof}

Next, we observe that $E_{5}$ is asymptotically negligible in mean.
%
\begin{lem}
Let the assumptions of Theorem~\ref{theo:clt_univariate} be satisfied.
Then
\begin{equation*}
\mathbb{E}|E_{5}| = n^{1/2} \mathbb{E}\Big | \left \langle (
\mathcal{G}_{n}^{-1 \ast } - \mathcal{G}^{-1 \ast }) v,
\widehat{u v_{1}} \right \rangle _{L^{2}(\mathbb{R})} \Big |
\to 0,
\quad \text{as} \ n \to \infty ,
\end{equation*}
for any $v \in \mathcal{U}(\xi , \beta _{2})$.
\end{lem}
\begin{proof}
From the proofs of Theorem~\ref{theo:upper_bound_estimation_error} and
Corollary~\ref{cor:order_of_convergence} we conclude that
\begin{equation*}
\sqrt{n} \ \mathbb{E}\Big | \left \langle (\mathcal{G}_{n}^{-1
\ast } - \mathcal{G}^{-1 \ast }) v, \widehat{uv_{1}} \right \rangle _{L^{2}(\mathbb{R})} \Big |
\lesssim \sqrt{n} \ a_{n}^{\frac{\beta
_{2}}{\beta _{1}} - 1} \Big ( \frac{n}{b_{n}} \Big )^{1/2} ;
\end{equation*}
hence, $\mathbb{E}|E_{5}| \to 0$ as $n \to \infty $, since $a_{n} = o
\Big ( \Big ( \frac{n}{\sqrt{b_{n}}} \Big )^{\frac{\beta _{1}}{\beta
_{1} - \beta _{2}}} \Big )$.
\end{proof}

\begin{lem}
Suppose $\gamma = 0$ and let the assumptions of
Theorem~\ref{theo:clt_univariate} be satisfied. Then
\begin{equation*}
\mathbb{E}|E_{4}| = \sqrt{n} \ \mathbb{E}\Big | \Big < \mathcal{F}
_{+}[\mathcal{G}^{-1 \ast }v], \mathcal{F}_{+}[K_{b}] \Big \{ \theta
\frac{\hat{\psi } - \psi }{\psi ^{2}} - \frac{\hat{\theta }}{\psi }
\Big \}
\Eins_{ \{ |\hat{\psi }| \leq n^{-1/2} \} } \Big >_{L^{2}(
\mathbb{R})} \Big | \to 0,
\end{equation*}
as $n \to \infty $, for any admissible function $v \in \mathcal{U}(
\xi , \beta _{2})$.
\end{lem}
\begin{proof}
Since $\frac{\theta }{\psi ^{2}} = i \Big ( \frac{1}{\psi } \Big )^{
\prime }$, we obtain by conditions \textbf{(K2)} and \textbf{(K3)},
\begin{equation*}
\begin{split}
\mathbb{E}|E_{4}|
\leq{}
& S \int _{-b_{n}^{-1}}^{b_{n}^{-1}} |
\mathcal{F}_{+}[\mathcal{G}^{-1 \ast }v](x)| \Big | \Big ( \frac{1}{
\psi } \Big )^{\prime }(x) \Big | \sqrt{n} \ \mathbb{E}\Big [ |
\hat{\psi }(x) - \psi (x)|
\Eins_{\{ |\hat{\psi }(x)| \leq n^{-1/2}
\}} \Big ] dx
\\
& + S \int _{-b_{n}^{-1}}^{b_{n}^{-1}} \frac{|\mathcal{F}_{+}[
\mathcal{G}^{-1 \ast }v](x)|}{|\psi (x)|} \sqrt{n} \ \mathbb{E}
\Big [ |\hat{\theta }(x)| \Eins_{\{ |\hat{\psi }(x)| \leq n^{-1/2}
\}} \Big ] dx,
\end{split}
\end{equation*}
with $S := \sup _{x \in \mathbb{R}, \ b > 0}|
\mathcal{F}_{+}[K_{b}](x)|$. In order to bound the summands on the
right-hand side of the latter inequality, we start with the following
observation: $\exists $ $n_{0} \in \mathbb{N}$ such that for all
$n \geq n_{0}$,
%
\begin{equation}
\label{eq:psi_bigger_n_half}
x \in [-b_{n}^{-1}, b_{n}^{-1}] \quad \Rightarrow \quad |\psi (x)| > 2n
^{-1/2}.\vadjust{\goodbreak}
\end{equation}
Indeed, by Lemma~\ref{lem:important_lemma}, (2), there is a constant
$c > 0$ such that $\frac{1}{|\psi (x)|} \leq c (1+|x|)^{\frac{1}{2} -
\varepsilon }$, for all $x \in \mathbb{R}$. Hence, if $|\psi (x)|
\leq 2n^{-1/2}$, then $|x| \geq \Big ( \frac{1}{2c} \Big )^{2/(1-2
\varepsilon )} n^{1/(1-2\varepsilon )} - 1$. Since $b_{n}^{-1} = o( n
^{1/(1-2\varepsilon )} )$ as $n \to \infty $, there exists $n_{0}
\in \mathbb{N}$, such that $b_{n}^{-1} < \Big ( \frac{1}{2c} \Big )^{2/(1-2
\varepsilon )} n^{1/(1-2\varepsilon )} - 1$ for all $n \geq n_{0}$, i.e.
$x \notin [-b_{n}^{-1}, b_{n}^{-1}]$. This
shows~\eqref{eq:psi_bigger_n_half}.

In the sequel, we assume that $n \geq n_{0}$ and consider each summand
in the above inequality seperately:
\begin{enumerate}%
\item
Using the  $m$-dependence of $(Y_{j})_{j \in \mathbb{Z}^{d}}$, we conclude
as in the first part of the proof of~\cite[Lemma 8.3]{KaRoSpoWalk18}
that
%
\begin{equation}
\label{eq:propb_estimation}
\mathbb{P}(|\hat{\psi }(x)| \leq n^{-1/2}) \lesssim \frac{n^{-p}}{|
\psi (x)|^{2p}},
\end{equation}
for any $p \geq 1/2$ and all $x \in \mathbb{R}$ with $|\psi (x)| > 2n
^{-1/2}$. Taking $p=1/2$ in~\eqref{eq:propb_estimation}, by the
\xch{Cauchy--Schwarz}{Cauchy-Schwart} inequality,~\cite[Lemma 8.2]{KaRoSpoWalk18} and
Lemma~\ref{lem:important_lemma}, (2), (3), we find that
\begin{equation*}
\begin{split}
& \int _{-b_{n}^{-1}}^{b_{n}^{-1}} |\mathcal{F}_{+}[\mathcal{G}^{-1
\ast }v](x)| \Big | \Big ( \frac{1}{\psi } \Big )^{\prime }(x) \Big |
\sqrt{n} \ \mathbb{E}\Big [ |\hat{\psi }(x) - \psi (x)|
\Eins_{\{
|\hat{\psi }(x)| \leq n^{-1/2} \}} \Big ] dx
\\
\lesssim{}
& \int _{-b_{n}^{-1}}^{b_{n}^{-1}} |\mathcal{F}_{+}[
\mathcal{G}^{-1 \ast }v](x)| \Big | \Big ( \frac{1}{\psi } \Big )^{
\prime }(x) \Big | \ \mathbb{P}\Big ( |\hat{\psi }(x)| \leq n^{-1/2}
\Big )^{1/2}
\Eins_{\{|\psi (x)| > 2n^{-1/2}\}} dx
\\
\leq{}
& n^{-1/4} \int _{-b_{n}^{-1}}^{b_{n}^{-1}} \frac{| \mathcal{F}
_{+}[\mathcal{G}^{-1 \ast }v](x)|}{|\psi (x)|^{1/2}} \Big | \Big ( \frac{1}{
\psi } \Big )^{\prime }(x) \Big |
\Eins_{\{|\psi (x)| > 2n^{-1/2}
\}} dx
\\
\leq{}
& n^{-1/4} \Big \| \frac{\mathcal{F}_{+}[\mathcal{G}^{-1 \ast }v]}{
\psi } \Big \|_{L^{1}(\mathbb{R})}
\Big \| \Big ( \frac{1}{\psi }
\Big )^{\prime }\Big \|_{L^{\infty }(\mathbb{R})},
\end{split}
\end{equation*}
for all $n \geq n_{0}$, where the last inequality uses the fact that
$|\psi (x)| \leq 1$. Hence, the first integral tends to zero as
$n \to \infty $.
\item
For the second integral, by the triangle inequality we observe that for any
$n \geq n_{0}$,
\begin{equation*}
\begin{split}
& \int _{-b_{n}^{-1}}^{b_{n}^{-1}} \frac{|\mathcal{F}_{+}[\mathcal{G}
^{-1 \ast }v](x)|}{|\psi (x)|} \sqrt{n} \ \mathbb{E}\Big [ |
\hat{\theta }(x)| \Eins_{\{|\hat{\psi }(x)|
\leq n^{-1/2}\}}
\Big ] dx
\\
\leq{}
& \int _{-b_{n}^{-1}}^{b_{n}^{-1}} \Big ( I_{1}(x) + I_{2}(x)
\Big ) dx,
\end{split}
\end{equation*}
where
\begin{equation*}
I_{1}(x) = \frac{|\mathcal{F}_{+}[\mathcal{G}^{-1 \ast }v](x)|}{|
\psi (x)|} \sqrt{n} \ \mathbb{E}\Big [ |\hat{\theta }(x) - \theta (x)|
\Eins_{|\hat{\psi }(x)| \leq n^{-1/2}} \Big ] \Eins_{\{|
\psi (x)| > 2n^{-1/2}\}}
\end{equation*}
and
\begin{equation*}
I_{2}(x) = |\mathcal{F}_{+}[\mathcal{G}^{-1 \ast }v](x)| \sqrt{n}
\ \frac{|\theta (x)|}{|\psi (x)|}
\mathbb{P}\Big ( |\hat{\psi }(x)|
\leq n^{-1/2} \Big ) \Eins_{\{|\psi (x)| > 2n^{-1/2}\}}.
\end{equation*}
Applying Lemma~\ref{lem:moment_bound_theta} with $q=1/2$ (cf.
Appendix~\ref{subsec:moment_inequalities}), we find that
\begin{equation*}
\begin{split}
I_{1}(x) \lesssim{}
& \sqrt{\mathbb{E}|Y_{0}|^{2}} \ \frac{|
\mathcal{F}_{+}[\mathcal{G}^{-1 \ast }v](x)|}{|\psi (x)|},
\quad \ \text{for all} \ x \in \mathbb{R};
\end{split}
\end{equation*}
hence, by Lemma~\ref{lem:important_lemma}, (3) and the finite
$(2+\tau )$-moment condition, $I_{1}$ is majorized by an integrable
function. Moreover, applying the  \xch{Cauchy--Schwarz}{Cauchy-Schwart}
inequality,~\eqref{eq:propb_estimation} and again
Lemma~\ref{lem:moment_bound_theta} (with $q=1$) yields
\begin{equation*}
\begin{split}
I_{1}(x) \lesssim{}
& \frac{|\mathcal{F}_{+}[\mathcal{G}^{-1 \ast }v](x)|}{|
\psi (x)|} \mathbb{P}\Big ( |\hat{\psi }(x)| \leq n^{-1/2} \Big )^{1/2}
\Eins_{\{|\psi (x)| > 2n^{-1/2}\}}
\\
\leq{}
&
\frac{|\mathcal{F}_{+}[\mathcal{G}^{-1 \ast }v](x)|}{|\psi (x)|} \frac{n
^{-1/4}}{|\psi (x)|^{1/2}} \to 0, \quad \text{as} \ n\to \infty ,
\end{split}
\end{equation*}
for all $x \in \mathbb{R}$. By dominated convergence, $
\lim _{n \to \infty } \int _{-b_{n}^{-1}}^{b_{n}^{-1}} I_{1}(x) dx = 0$
follows. Further,
\begin{equation*}
\begin{split}
\int _{-b_{n}^{-1}}^{b_{n}^{-1}} I_{2}(x) dx \leq{}
& n^{-1/2}
\int _{-b_{n}^{-1}}^{b_{n}^{-1}} \frac{|\mathcal{F}_{+}[\mathcal{G}
^{-1 \ast }v](x)|}{|\psi (x)|} \frac{|\theta (x)|}{|\psi (x)|^{2}}
\Eins_{\{|\psi (x)| > 2n^{-1/2}\}} dx
\\
\leq{}
& n^{-1/2} \Big \| \frac{\mathcal{F}_{+}[\mathcal{G}^{-1 \ast }v]}{
\psi } \Big \|_{L^{1}(\mathbb{R})} \Big \| \Big ( \frac{1}{\psi }
\Big )^{\prime }\Big \|_{L^{\infty }(\mathbb{R})},
\end{split}
\end{equation*}
by~\eqref{eq:propb_estimation} (with $p=1$),
Lemma~\ref{lem:important_lemma}, (2), (3) and the Cauchy--Schwarz inequality;
hence, also $\lim _{n \to \infty } \int _{-b_{n}^{-1}}^{b_{n}^{-1}} I
_{2}(x) dx = 0$.
\end{enumerate}
All in all, this shows the assertion of the lemma.
\end{proof}

\subsection{Main stochastic term\index{main stochastic term}}

In this section we show the asymptotic normality of the main stochastic
term.\index{main stochastic term} For this purpose, let $P_{n}:\mathcal{B}(\mathbb{R}) \to [0,1]$
be the empirical measure given by
\begin{equation*}
P_{n} = \frac{1}{n} \sum _{j \in W} \delta _{Y_{j}},
\end{equation*}
where $\delta _{x}:\mathcal{B}(\mathbb{R}) \to \{ 0,1 \}$ denotes the
\xch{Dirac}{dirac} measure concentrated in $x \in \mathbb{R}$. Further, for any
$v \in \mathcal{U}(\xi , \beta _{2})$, define the random fields\index{random ! field}
$(Z_{j,v,n}^{(k)})_{j \in \mathbb{Z}^{d}}$, $k=1,2$, by
%
\begin{equation}
\label{eq:Z_j_v_n_1}
Z_{j,v,n}^{(1)} = \frac{1}{2\pi } Y_{j} \mathcal{F}_{+} \Big [ \frac{
\mathcal{F}_{+}[\mathcal{G}^{-1 \ast }v](-\, \cdot \,)}{\psi }
\mathcal{F}_{+}[K_{b_{n}}] \Big ](Y_{j})
\end{equation}
and
%
\begin{equation}
\label{eq:Z_j_v_n_2}
Z_{j,v,n}^{(2)} = \frac{i}{2\pi } \mathcal{F}_{+} \Big [ \mathcal{F}
_{+}[\mathcal{G}^{-1 \ast }v](-\, \cdot \,) \Big ( \frac{1}{\psi }
\Big )^{\prime }\mathcal{F}_{+}[K_{b_{n}}] \Big ](Y_{j}).
\end{equation}
The following theorem is the main result of this section.

\begin{theo}
\label{theo:asymptotic_normality_of_main_stoch_term}
Let the assumptions of Theorem~\ref{theo:clt_univariate} be satisfied.
Then, as $W$ is regularly growing\index{regularly growing} to infinity,
\begin{equation*}
\frac{1}{2\pi } E_{1} = \frac{\sqrt{n}}{2\pi } \Big < \mathcal{F}
_{+}[\mathcal{G}^{-1 \ast }v], \Big \{ \frac{\hat{\theta } -
\theta }{
\psi } - i \Big ( \frac{1}{\psi } \Big )^{\prime }(\hat{\psi } - \psi )
\Big \} \mathcal{F}_{+}[K_{b}] \Big >_{L^{2}(\mathbb{R})}
\overset{d}{\to } N_{v},
\end{equation*}
for any $v \in \mathcal{U}(\xi , \beta _{2})$, where $N_{v}$ is a
Gaussian random variable\index{Gaussian random variable} with zero expectation and variance
$\sigma ^{2}$ given in~\eqref{eq:asymptotic_variance}.
\end{theo}

In order to prove
Theorem~\ref{theo:asymptotic_normality_of_main_stoch_term}, we first
show some auxiliary statements. We begin with the following
representation for the main stochastic term.\index{main stochastic term}

\begin{lem}
\label{lem:representation_of_main_stoch_term}
Let $v \in \mathcal{U}(\xi , \beta _{2})$. Then, under the assumptions
of Theorem~\ref{theo:clt_univariate}, the main stochastic term\index{main stochastic term} can be
represented by
\begin{equation*}
\frac{1}{2\pi } E_{1} = \frac{1}{\sqrt{n}} \sum _{j \in W} \Big ( Z
_{j,v,n}^{(1)} - Z_{j,v,n}^{(2)} \Big ),
\end{equation*}
with $Z_{j,v,n}^{(k)}$, $k=1,2$, given in~\eqref{eq:Z_j_v_n_1}
and~\eqref{eq:Z_j_v_n_2}.
\end{lem}

\begin{proof}
Since $\theta = -i \psi ^{\prime }$,
\begin{equation*}
i \Big ( \psi \Big ( \frac{1}{\psi } \Big )^{\prime }- \frac{\theta }{
\psi } \Big ) = i \Big ( \psi \frac{1}{\psi } \Big )^{\prime }= 0;
\end{equation*}
hence,
\begin{equation*}
\begin{split}
\frac{1}{2\pi } E_{1} ={}
& \frac{\sqrt{n}}{2\pi } \Big < \mathcal{F}
_{+}[\mathcal{G}^{-1 \ast }v], \Big \{
\frac{\hat{\theta }}{\psi } - i
\Big ( \frac{1}{\psi } \Big )^{\prime }\hat{\psi } \Big \} \mathcal{F}
_{+}[K_{b_{n}}] \Big >_{L^{2}(\mathbb{R})}
\\
={}
& \frac{\sqrt{n}}{2\pi } \int _{\mathbb{R}} \mathcal{F}_{+}[
\mathcal{G}^{-1 \ast }v](x) \Big \{
\frac{\hat{\theta }(-x)}{\psi (-x)} - i \Big ( \frac{1}{\psi } \Big )^{
\prime }(-x) \hat{\psi }(-x) \Big \} \mathcal{F}_{+}[K_{b_{n}}](-x) dx
\\
={}
& \frac{\sqrt{n}}{2\pi } \Big [ \int _{\mathbb{R}} \mathcal{F}_{+}[
\mathcal{G}^{-1 \ast }v](x) \frac{\hat{\theta }(-x)}{\psi (-x)}
\mathcal{F}_{+}[K_{b_{n}}](-x) dx
\\
& - i \int _{\mathbb{R}} \mathcal{F}_{+}[\mathcal{G}^{-1 \ast }v](x)
\Big ( \frac{1}{\psi } \Big )^{\prime }(-x) \hat{\psi }(-x) \mathcal{F}
_{+}[K_{b_{n}}](-x) dx \Big ].
\end{split}
\end{equation*}
Now, taking into account that $\hat{\psi }(x) = \int _{\mathbb{R}} e
^{itx} P_{n}(dt)$ and $\hat{\theta }(x) = \int _{\mathbb{R}} e^{itx} t
P_{n}(dt)$, Fubini's theorem yields the desired result.
\end{proof}

The following lemma justifies the asymptotic variance $\sigma ^{2}$ in
Theorem~\ref{theo:clt_univariate}.

\begin{lem}
\label{lem:covariance_convergence_lemma}
Let the assumptions of Theorem~\ref{theo:clt_univariate} be satisfied
and suppose functions $v_{1} \in \mathcal{U}(\xi _{1}, \beta _{2}^{(1)})$
and $v_{2} \in \mathcal{U}(\xi _{2}, \beta _{2}^{(2)})$. Then, as $W$ is
regularly growing\index{regularly growing} to infinity,
\begin{equation*}
\begin{split}
\Cov \Big ( |W|^{-1/2} \sum _{j \in W} \Big ( Z_{j,v_{1},n}^{(1)} - Z
_{j,v_{1},n}^{(2)} \Big ),
|W|^{-1/2} \sum _{j \in W} \Big ( Z_{j,v_{2},n}
^{(1)} - Z_{j,v_{2},n}^{(2)} \Big ) \Big )
\to \sigma _{v_{1}, v_{2}},
\end{split}
\end{equation*}
with $\sigma _{v_{1}, v_{2}} \in \mathbb{R}$ given by
\begin{equation*}
\sigma _{v_{1}, v_{2}} =
\sum _{\substack{t \in \mathbb{Z}^{d}:
\\
\|t\|_{\infty }\leq m}} \mathbb{E}\Big [ \Big ( Z_{t,v_{1}}^{(1)} - Z
_{t,v_{1}}^{(2)} \Big )
\Big ( Z_{0,v_{2}}^{(1)} - Z_{0,v_{2}}^{(2)}
\Big ) \Big ].
\end{equation*}
\end{lem}

\begin{proof}
Let $v_{1} \in \mathcal{U}(\xi _{1}, \beta _{2}^{(1)})$, $v_{2} \in
\mathcal{U}(\xi _{2}, \beta _{2}^{(2)})$ and define the functions
$g^{(k)}$, $g_{n}^{(k)}: \mathbb{R}\to \mathbb{R}$, $k=1,2$, by
\begin{equation*}
\begin{split}
g_{n}^{(k)}(x) ={}
& \frac{x}{2\pi } \mathcal{F}_{+} \Big [ \frac{
\mathcal{F}_{+}[\mathcal{G}^{-1 \ast }v_{k}](-\, \cdot \,)}{\psi }
\mathcal{F}_{+}[K_{b_{n}}] \Big ](x)
\\
&- \frac{i}{2\pi } \mathcal{F}_{+} \Big [ \mathcal{F}_{+}[\mathcal{G}
^{-1 \ast }v_{k}](-\, \cdot \,) \Big ( \frac{1}{\psi } \Big )^{\prime }
\mathcal{F}_{+}[K_{b_{n}}] \Big ](x),
\quad x \in \mathbb{R},
\end{split}
\end{equation*}
and
\begin{equation*}
\begin{split}
g^{(k)}(x) ={}
& \frac{x}{2\pi } \mathcal{F}_{+} \Big [ \frac{
\mathcal{F}_{+}[\mathcal{G}^{-1 \ast }v_{k}](-\, \cdot \,)}{\psi }
\Big ](x)
\\
&- \frac{i}{2\pi } \mathcal{F}_{+} \Big [ \mathcal{F}_{+}[\mathcal{G}
^{-1 \ast }v_{k}](-\, \cdot \,) \Big ( \frac{1}{\psi } \Big )^{\prime }
\Big ](x),
\quad x \in \mathbb{R}.
\end{split}
\end{equation*}
Then $(g^{(k)}(Y_{j}))_{j \in \mathbb{Z}^{d}}$ and $(g_{n}^{(k)}(Y
_{j}))_{j \in \mathbb{Z}^{d}}$ fulfill properties (1)--(3) from
Lemma~\ref{lem:conv_of_covariance_fct_cubes_version} (cf.
Appendix~\ref{subsec:asymptotic_variance}). Indeed, by
Lemma~\ref{lem:representation_of_main_stoch_term}, it follows that
\begin{equation*}
\begin{split}
\mathbb{E}[ g_{n}^{(k)}(Y_{0}) ] ={}
& \mathbb{E}\Big [ Z_{0,v_{k},n}
^{(1)} - Z_{0,v_{k},n}^{(2)} \Big ] =
\frac{1}{n}\mathbb{E}\Big [
\sum _{j \in W} \Big ( Z_{j,v_{k},n}^{(1)} - Z_{j,v_{k},n}^{(2)} \Big )
\Big ]
\\
={}
& \frac{1}{2\pi } \Big < \mathcal{F}_{+}[\mathcal{G}^{-1 \ast }v],
\Big \{ \frac{\hat{\theta } -
\theta }{\psi } - i \Big ( \frac{1}{
\psi } \Big )^{\prime }(\hat{\psi } - \psi ) \Big \} \mathcal{F}_{+}[K
_{b_{n}}] \Big >_{L^{2}(\mathbb{R})}.
\end{split}
\end{equation*}
Since $\mathbb{E}[\hat{\psi }(x) - \psi (x)] = \mathbb{E}[
\hat{\theta }(x) - \theta (x)] = 0$ for all $x \in \mathbb{R}$, we
conclude by Fubini's theorem that $\mathbb{E}[ g_{n}^{(k)}(Y_{0}) ] =
\mathbb{E}[ g^{(k)}(Y_{0}) ] = 0$, for $k=1,2$. Moreover, since the
Fourier transform of an integrable function is bounded, the finite
$(2+\tau )$-moment condition together with
Lemma~\ref{lem:important_lemma}, (2), (3) and \textbf{(K3)} imply
$\mathbb{E}| g^{(k)}(Y_{0}) |^{2}$, $\mathbb{E}| g_{n}^{(k)}(Y_{0})
|^{2} < \infty $, $k=1,2$. The same arguments in combination with
the dominated convergence yields
\begin{equation*}
\mathbb{E}\Big [ g_{n}^{(1)}(Y_{0}) g_{n}^{(2)}(Y_{j}) \Big ] \to
\mathbb{E}\Big [ g^{(1)}(Y_{0}) g^{(2)}(Y_{j}) \Big ],
\end{equation*}
as $|W| \to \infty $. Hence,
Lemma~\ref{lem:conv_of_covariance_fct_cubes_version} yields the
assertion of the lemma.
\end{proof}

We now can give a proof of
Theorem~\ref{theo:asymptotic_normality_of_main_stoch_term}.

\begin{proof}[Proof of Theorem~\ref{theo:asymptotic_normality_of_main_stoch_term}]
If $\sigma _{v}^{2} = 0$, then
Lemma~\ref{lem:covariance_convergence_lemma} yields
\begin{equation*}
\sigma _{v,n}^{2} := \mathbb{E}\Big [ \Big ( \frac{1}{\sqrt{n}}
\sum _{j \in W} \Big ( Z_{0,v,n}^{(1)} - Z_{0,v,n}^{(2)} \Big ) \Big )^{2}
\Big ] \to \sigma _{v}^{2} = 0,
\end{equation*}
as $W$ is regularly growing\index{regularly growing} to infinity; hence, $n^{-1/2} \sum _{j
\in W} \Big ( Z_{0,v,n}^{(1)} - Z_{0,v,n}^{(2)} \Big ) \to 0$ in
probability.

Now, assume that $\sigma _{v}^{2} > 0$ and choose $n_{0} \in
\mathbb{N}$ such that $\sigma _{v,n}^{2} > 0$ for all $n \geq n_{0}$
(which is indeed possible, since $\sigma _{v,n}^{2} \to \sigma _{v}^{2}
> 0$ as $n \to \infty $). For any $n \geq n_{0}$, let
\begin{equation*}
X_{j,n} := \frac{1}{\sqrt{n}} \frac{Z_{j,v,n}^{(1)} - Z_{j,v,n}^{(2)}}{
\sigma _{v,n}}, \quad j \in \mathbb{Z}^{d},
\end{equation*}
and denote by $F_{n}$ the distribution function of $\sum _{j \in W} X
_{j,n}$. In the proof of\break  Lemma~\ref{lem:covariance_convergence_lemma}
we have seen that $(X_{j,n})_{j \in \mathbb{Z}^{d}}$ is a centered
$m$-dependent random field and\break  $ \mathbb{E}|X_{j,n}|^{2+\tau } \leq c
n^{-1 - \tau /2} \sigma _{v,n}^{-(2+\tau )}$ for any $n \in \mathbb{N}$
and a constant $c > 0$. Hence, applying~\cite[Theorem 2.6]{Chen2004}
with $p = 2+\tau $ yields
\begin{equation*}
\sup \limits _{x \in \mathbb{R}} |F_{n}(x) - \phi (x)| \leq 75 c (m+1)^{(1
+ \tau )d} \sigma _{v,n}^{-(2+\tau )} n^{-\tau / 2} \to 0,
\end{equation*}
as $n \to \infty $. This completes the proof.
\end{proof}

\subsection{Remainder term}

In this section, we show that the remainder term\index{remainder term} $E_{2}$ is
stochastically negligible as the sample size $n$ tends to infinity.

\begin{theo}
\label{theo:remainder_negligible_theorem}
Let $\gamma = 0$ and suppose that the assumptions of
Theorem~\ref{theo:clt_univariate} are satisfied. Then, as $n \to
\infty $,
\begin{equation*}
E_{2} = \sqrt{n} \Big < \mathcal{F}_{+}[\mathcal{G}^{-1 \ast }v],
\Big \{ R_{n} + \theta \frac{\psi - \hat{\psi }}{\psi ^{2}} \Eins
_{ \{ |\hat{\psi }| \leq n^{-1/2} \} } \Big \} \mathcal{F}_{+}[K_{b}]
\Big >_{L^{2}(\mathbb{R})}
\overset{\mathbb{P}}{\to } 0,
\end{equation*}
for any $v \in \mathcal{U}(\xi ,\beta _{2})$.
\end{theo}

In order to prove Theorem~\ref{theo:remainder_negligible_theorem}, some
auxiliary statements are required. Therefore, we introduce the following
notation.

For any $t \in \mathbb{R}$, $j \in \mathbb{Z}^{d}$, let the centered
random variables $\xi _{j}^{(l)}(t)$, $\tilde{\xi }_{j}^{(l)}(t)$,
$l=1,2$, be given by
\begin{equation*}
\begin{split}
\xi _{j}^{(1)}(t)
& = \cos (tY_{j}) - \mathbb{E}\Big [ \cos (tY_{0})
\Big ],
\\
\xi _{j}^{(2)}(t)
& = \sin (tY_{j}) - \mathbb{E}\Big [ \sin (tY_{0})
\Big ],
\\
\tilde{\xi }_{j}^{(1)}(t)
& = Y_{j} \cos (tY_{j}) - \mathbb{E}\Big [ Y
_{0} \cos (tY_{0}) \Big ],
\\
\tilde{\xi }_{j}^{(2)}(t)
& = Y_{j} \sin (tY_{j}) - \mathbb{E}\Big [ Y
_{0} \sin (tY_{0}) \Big ].
\end{split}
\end{equation*}
Then $\hat{\psi } - \psi $ and $\hat{\theta } - \theta $ can be
rewritten as
\begin{equation*}
\begin{split}
\hat{\psi }(t) - \psi (t) ={}
& \frac{1}{n} \sum _{j \in W} \Big ( \xi
_{j}^{(1)}(t) + i \xi _{j}^{(2)}(t) \Big ) \quad \text{and}
\\
\hat{\theta }(t) - \theta (t) ={}
& \frac{1}{n} \sum _{j \in W} \Big (
\tilde{\xi }_{j}^{(1)}(t) + i \tilde{\xi }_{j}^{(2)}(t) \Big ).
\end{split}
\end{equation*}
In the sequel, we shortly write $\xi ^{(l)}(t)$, $\tilde{\xi }^{(l)}(t)$
for the random fields\index{random ! field} $(\xi _{j}^{(l)}(t))_{j \in \mathbb{Z}^{d}}$ and
$(\tilde{\xi }_{j}^{(l)}(t))_{j \in \mathbb{Z}^{d}}$, $l=1,2$. Moreover,
for any $K > 0$, we define the random fields\index{random ! field} $\bar{\xi }_{K}^{(l)}(t)
= (\bar{\xi }_{j,K}^{(l)}(t))_{j \in \mathbb{Z}^{d}}$, $l=1,2$, by
\begin{equation*}
\begin{split}
\bar{\xi }_{j,K}^{(1)}(t)
& = Y_{j} \cos (tY_{j}) \Eins_{[-K,K]}(Y
_{j}) - \mathbb{E}\Big [ Y_{0} \cos (tY_{0}) \Eins_{[-K,K]}(Y_{0})
\Big ],
\\
\bar{\xi }_{j,K}^{(2)}(t)
& = Y_{j} \sin (tY_{j}) \Eins_{[-K,K]}(Y
_{j}) - \mathbb{E}\Big [ Y_{0} \sin (tY_{0}) \Eins_{[-K,K]}(Y_{0})
\Big ].
\end{split}
\end{equation*}
For any finite subset $V \subset \mathbb{Z}^{d}$ and any random field
$Y = (Y_{j})_{j \in \mathbb{Z}^{d}}$, let
\begin{equation*}
S_{V}(Y) = \sum _{j \in V} Y_{j}.
\end{equation*}

\begin{lem}
\label{lem:pconvrate}
Let the assumptions of Theorem~\ref{theo:clt_univariate} be satisfied
and suppose $K \geq 1$. Then
%
\begin{equation}
\label{eq:1ststatement}
\mathbb{P}(|S_{W}(\xi ^{(l)}(t))| \geq x) \leq 2 \exp \left ( -
\frac{1}{8(m+1)^{d}} \frac{x^{2}}{x + 2 |W|} \right )
\end{equation}
and
%
\begin{equation}
\label{eq:2ndstatement}
\mathbb{P}(|S_{W}(\bar{\xi }_{K}^{(l)}(t))| \geq x) \leq 2 \exp \left (
-\frac{1}{8(m+1)^{d} K^{2}} \frac{x^{2}}{x + 2 |W|} \right ),
\end{equation}
for any $t \in \mathbb{R}$, $x \geq 0$ and $l=1,2$.
\end{lem}
\begin{proof}
Since $|\xi _{j}^{(l)}(t)| \leq 2$ for all $t \in \mathbb{R}$,
$j \in \mathbb{Z}^{d}$ and $l=1,2$, we have that
\begin{align*}
|\mathbb{E}[\xi _{j}^{(l)}(t)^{p}] |
& \leq \mathbb{E}[
|\xi _{j}^{(l)}(t)|^{p-2} \xi _{j}^{(l)}(t)^{2} ]
\leq 2^{p-2}
\mathbb{E}[ \xi _{j}^{(l)}(t)^{2} ], \quad p=3,4,\dots ;
\end{align*}
hence, Theorem~\ref{theo:bernstein} (with $H=2$)
implies~\eqref{eq:1ststatement}. Next, we obtain
\begin{align*}
|\mathbb{E}[ \bar{\xi }_{j,K}^{(l)}(t)^{p} ] | \leq \mathbb{E}[ |\bar{
\xi }_{j,K}^{(l)}(t)|^{p-2} |\bar{\xi }_{j,K}^{(l)}(t)|^{2} ]
\leq (2K)^{p-2} \mathbb{E}[ \bar{\xi }_{j,K}^{(l)}(t)^{2} ], \quad p=3,4,
\dots .
\end{align*}
Taking into account that $\sum _{j\in W} \mathbb{E}[ \bar{\xi }_{j,K}
^{(l)}(t)^{2} ] \leq 4 K^{2} |W|$, Theorem~\ref{theo:bernstein} (with
$H=2K$) yields the bound in~\eqref{eq:2ndstatement}.
\end{proof}

\begin{lem}
\label{lem:conv_rate}
Let the assumptions of Theorem~\ref{theo:clt_univariate} be satisfied
and let $n = |W|$. Moreover, for any $n$, suppose numbers $\varepsilon _{n} > 0$, $K_{n} \geq 1$, such that $\varepsilon _{n} \to 0$ and $K_{n} \to
\infty $ as $n \to \infty $. Then, for any $n$ with $\varepsilon _{n} <
\min \{ 1, \frac{T}{4} \}$,
%
\begin{align}
\label{eq:est_phi}
\mathbb{P}\Big ( \sup _{t \in [-T,T]} n^{-1} |S_{W}(\xi ^{(l)}(t))|
\geq \varepsilon _{n} \Big )
& \leq C_{1} \sqrt{ \frac{T}{\varepsilon
_{n}}} \exp \left \lbrace -
\frac{n \varepsilon _{n}^{2}}{160(m+1)^{d}} \right \rbrace \end{align}
and
%
\begin{equation}
\label{eq:est_phi_prime}
\begin{split}
\mathbb{P}\Big (\sup _{t \in [-T,T]} n^{-1}|S_{W}(\tilde{\xi }^{(l)}(t))|
\geq \varepsilon _{n} \Big ) \leq{}
&
C_{2} \sqrt{ \frac{T}{\varepsilon
_{n}}} \exp \left \lbrace - \frac{n \varepsilon _{n}^{2}}{576(m+1)^{d}
K_{n}^{2}} \right \rbrace \\
& + \frac{C_{3}}{\varepsilon _{n} K_{n}^{1+\tau }},
\end{split}
\end{equation}
$l =1,2$, where $C_{1} = 4(1+2\mathbb{E}|Y_{0}\xch{|)}{|}$, $C_{2} = 4
\sqrt{2}(1+2\mathbb{E}|Y_{0}|^{2} )$ and $C_{3} = 8 \mathbb{E}|Y_{0}|^{2+
\tau }$.
\end{lem}
\begin{proof}
We use the same idea as in the proof of~\cite[Theorem
2]{BelPanWoern16}: divide the interval $[-T,T]$ by $2J$ equidistant
points $(t_{k})_{k=1,\dots ,2J} = \mathcal{D}$, where $t_{k} = -T + k
\frac{T}{J}$, $k=1,\dots ,2J$. Then, for any $t \in [-T,T]$ such that
$|t-t_{k}| \leq \frac{T}{J}$, we have for any $j \in \mathbb{Z}^{d}$
that
\begin{align*}
|\xi ^{(l)}_{j}(t)-\xi ^{(l)}_{j}(t_{k})| \leq |t-t_{k}| (|Y_{j}|+
\mathbb{E}|Y_{0}|) \leq (|Y_{j}|+\mathbb{E}|Y_{0}|) \frac{T}{J},
\quad l = 1,2.
\end{align*}
Hence, by Markov's inequality and Lemma~\ref{lem:pconvrate}, for any
$n \in \mathbb{N}$, we obtain that
\begingroup
\abovedisplayskip=11pt
\belowdisplayskip=11pt
\begin{align*}
& \mathbb{P}\Bigl(\sup _{t \in [-T,T]} n^{-1} |S_{W}(\xi ^{(l)}(t))| \geq
\varepsilon _{n}\Bigr) = \mathbb{P}\Bigl(\sup _{t_{k} \in \mathcal{D}}
\sup _{t: |t-t_{k}| \leq \frac{T}{J}}| S_{W}(\xi ^{(l)}(t))| \geq n
\varepsilon _{n}\Bigr)
\\
\leq{}
& \mathbb{P}\Bigl(\sup _{t_{k} \in \mathcal{D}} |S_{W}(\xi ^{(l)}(t
_{k}))| \geq \frac{n \varepsilon _{n}}{2}\Bigr)
\\
& + \mathbb{P}\Bigl(\sup _{t_{k} \in \mathcal{D}}
\sup _{t:|t-t_{k}| \leq \frac{T}{J}}| S_{W}(\xi ^{(l)}(t))-S_{W}(\xi
^{(l)}(t_{k}))| \geq \frac{n \varepsilon _{n}}{2}\Bigr)
\\
\leq{}
& \sum _{t_{k} \in \mathcal{D}} \mathbb{P}\Bigl( |S_{W}(\xi ^{(l)}(t
_{k}))| \geq \frac{n \varepsilon _{n}}{2}\Bigr)
+ \mathbb{P}\Bigl(\sum _{j \in W}
(|Y_{j}|+\mathbb{E}|Y_{0}|) \frac{T}{J} \geq
\frac{n \varepsilon _{n}}{2}\Bigr)
\\
\leq{}
& 4J \exp \left \lbrace -\frac{1}{16(m+1)^{d}} \frac{n
\varepsilon _{n}^{2}}{\varepsilon _{n} + 4} \right \rbrace + \frac{4T }{J
\varepsilon _{n}} \mathbb{E}|Y_{0}|,
\end{align*}
$l=1,2$. Now, let $n \in \mathbb{N}$ be such that $\varepsilon _{n} <
\frac{T}{4}$ and choose
\begin{equation*}
J = \Bigg \lfloor \left ( \frac{T}{\varepsilon _{n}}
\exp \left \lbrace \frac{1}{16(m+1)^{d}} \frac{n \varepsilon _{n}^{2}}{
\varepsilon _{n} + 4} \right \rbrace \right )^{1/2} \Bigg \rfloor ,
\end{equation*}
where $\lfloor x \rfloor $ denotes the \xch{integer}{integere} part of $x \in
\mathbb{R}$. Then
\begin{align*}
\mathbb{P}\Bigl ( \sup _{t \in [-T,T]} n^{-1} |S_{W}(\xi ^{(l)}(t))|
\geq \varepsilon _{n} \Bigr) \leq{}
& C_{1}
\sqrt{ \frac{T}{\varepsilon
_{n}}} \exp \left \lbrace - \frac{1}{32(m+1)^{d}} \frac{n \varepsilon
_{n}^{2}}{\varepsilon _{n} + 4} \right \rbrace
\end{align*}
with $C_{1} = 4(1+2\mathbb{E}|Y_{0}| )$. Applying the same arguments to
$\sup _{t \in [-T,T]}|S_{W}(\bar{\xi }_{K_{n}}^{(l)}(t)\xch{)|}{|}$ yields
\begin{align*}
\mathbb{P}\Bigl ( \sup _{t \in [-T,T]} n^{-1} |S_{W}(\bar{\xi }_{K_{n}}
^{(l)}(t))| \geq \varepsilon _{n}  \Bigr) \leq \tilde{C}
\sqrt{ \frac{T}{
\varepsilon _{n}}} \exp \left \lbrace -
\frac{1}{32(m+1)^{d} K_{n}^{2}} \frac{n \varepsilon _{n}^{2}}{
\varepsilon _{n} + 4} \right \rbrace ,
\end{align*}
whenever $\varepsilon _{n} < \frac{T}{4}$, where $\tilde{C} = 4(1+2
\mathbb{E}|Y_{0}|^{2} )$. Combining Markov's inequality, H\"{o}lder's
inequality and the finite $(2+\tau )$-moment property of $Y_{0}$ implies
\begin{align*}
&\mathbb{P}\Biggl(\sup _{t \in [-T,T]} n^{-1} \Big | \sum _{j \in W} \Big (
\tilde{\xi }_{j}^{(l)}(t) - \bar{\xi }_{j,K_{n}}^{(l)}(t) \Big )
\Big | \geq \frac{\varepsilon _{n}}{2}\Biggr)
\\
\leq{}
& \mathbb{P}\Biggl( \sum _{j \in W} (|Y_{j}|\Eins_{ ( K_{n},
\infty ) }(|Y_{j}|) + \mathbb{E}|Y_{0}|\Eins_{( K_{n}, \infty )}(|Y
_{0}|) ) \geq \frac{n \varepsilon _{n}}{2}\Biggr)
\\
\leq{}
& \frac{4}{\varepsilon _{n}} \left ( \mathbb{E}|Y_{0}|^{2+\tau }
\right )^{1/(2+\tau )} \mathbb{P}(|Y_{0}| > K_{n})^{(1+\tau )/(2+
\tau )}
\\
\leq{}
& \frac{4}{K_{n}^{1+\tau } \varepsilon _{n}} \mathbb{E}|Y_{0}|^{2+
\tau },
\end{align*}
\endgroup
$l=1,2$. All in all, we have for any $n$ such that $\varepsilon _{n} <
\frac{T}{2}$,
\begin{align*}
& \mathbb{P}\Bigl(\sup _{t \in [-T,T]} n^{-1} |S_{W}(\tilde{\xi }^{(l)}(t))|
\geq \varepsilon _{n}\Bigr)
\\
\leq{}
& \mathbb{P}\Bigl(\sup _{t \in [-T,T]} n^{-1} |S_{W}(\bar{\xi }_{K
_{n}}^{(l)}(t))| \geq \frac{\varepsilon _{n}}{2}\Bigr)
\\
&+ \mathbb{P}\Bigl(\sup _{t \in [-T,T]} n^{-1} |S_{W}(\tilde{\xi }^{(l)}(t)
- \bar{\xi }_{K_{n}}^{(l)}(t))| \geq \frac{\varepsilon _{n}}{2}\Bigr)
\\
\leq{}
& \sqrt{2} \tilde{C} \sqrt{ \frac{T}{\varepsilon _{n}}}
\exp \left \lbrace - \frac{1}{64(m+1)^{d} K_{n}^{2}} \frac{n
\varepsilon _{n}^{2}}{\varepsilon _{n} + 8} \right \rbrace + \frac{8}{K
_{n}^{1+\tau } \varepsilon _{n}} \mathbb{E}|Y_{0}|^{2+\tau }.
\end{align*}
Hence, it follows for any $n$ with $\varepsilon _{n} < \min \{ 1,
\frac{T}{4} \}$ that
\begin{align*}
\mathbb{P}\Bigl(\sup _{t \in [-T,T]} n^{-1}|S_{W}(\xi ^{(l)}(t))| \geq
\varepsilon _{n}\Bigr)
& \leq C_{1}
\sqrt{ \frac{T}{\varepsilon _{n}}}
\exp \left \lbrace - \frac{n\varepsilon _{n}^{2}}{160(m+1)^{d}} \right \rbrace \end{align*}
as well as
\begin{align*}
\mathbb{P}\Bigl(\sup _{t \in [-T,T]} n^{-1}|S_{W}(\tilde{\xi }_{j}^{(l)}(t))|
\geq \varepsilon _{n}\Bigr) \leq{}
& C_{2}
\sqrt{ \frac{T}{\varepsilon _{n}}} \exp \left \lbrace - \frac{n
\varepsilon _{n}^{2}}{576(m+1)^{d} K_{n}^{2}} \right \rbrace \\
& + \frac{C_{3}}{K_{n}^{1+\tau } \varepsilon _{n}},
\end{align*}
where $C_{2} = \sqrt{2} \tilde{C}$ and $C_{3} = 8 \mathbb{E}|Y_{0}|^{2+
\tau }$.
\end{proof}

\begin{theo}
\label{theo:conv_rate}
For some $\zeta > 0$, suppose $\varepsilon _{n} \approx n^{-\frac{1+
\tau }{2(2+\tau )}} \Big [ \log \Big ( T^{\frac{1}{2}} n^{\frac{1+
\tau }{4(2+\tau )}} \Big ) \Big ]^{\zeta + \frac{1}{2}}$ and $K_{n} = n
^{\frac{1}{2(2+\tau )}}$ in Lemma~\ref{lem:conv_rate}. Then, for $n$
sufficiently large,
\begin{align*}
\mathbb{P}\Big ( \max \Big \{ \sup _{t \in [-T,T]} |\hat{\psi }(t) -
\psi (t)|, \sup _{t \in [-T,T]}
|\hat{\theta }(t) - \theta (t)| \Big
\} > \varepsilon _{n} \Big ) \leq \tilde{C} y_{n},
\end{align*}
where
\begin{equation*}
y_{n} = \Big [ \log \Big ( T^{\frac{1}{2}} n^{
\frac{1+\tau }{4(2+\tau )}} \Big ) \Big ]^{-\frac{\zeta }{2}-
\frac{1}{4}}
\end{equation*}
and $\tilde{C} > 0$ is a constant (independent of $T$).
\end{theo}
\begin{proof}
By Lemma~\ref{lem:conv_rate} it follows that
\begingroup
\abovedisplayskip=5pt
\belowdisplayskip=5pt
\begin{align*}
& \mathbb{P}\Big ( \max \Big \{ \sup _{t \in [-T,T]} |\hat{\psi }(t) -
\psi (t)|, \sup _{t \in [-T,T]}
|\hat{\theta }(t) - \theta (t)| \Big
\} > \varepsilon _{n} \Big )
\\
\leq{}
& \sum _{l=1}^{2}\mathbb{P}\Big ( \sup _{t \in [-T,T]} n^{-1}|S
_{W}(\xi ^{(l)}(t))| \geq \varepsilon _{n} \Big )
\\
& + \sum _{l=1}^{2} \mathbb{P}\Big ( \sup _{t \in [-T,T]} n^{-1}|S_{W}(
\tilde{\xi }^{(l)}(t))| \geq \varepsilon _{n} \Big )
\\
\leq{}
& C \left ( \sqrt{ \frac{T}{\varepsilon _{n}}}
\exp \left \lbrace - \frac{n \varepsilon _{n}^{2}}{576(m+1)^{d} K_{n}
^{2}} \right \rbrace + \frac{1}{\varepsilon _{n} K_{n}^{1+\tau }}
\right ),
\end{align*}
for some constant $C > 0$.  Let $\varepsilon _{n} = n^{-\frac{1+
\tau }{2(2+\tau )}} \Big [ \log \Big ( T^{\frac{1}{2}} n^{\frac{1+
\tau }{4(2+\tau )}} \Big ) \Big ]^{\zeta + \frac{1}{2}}$, without loss of generality. Then we observe
that
\begin{equation*}
\begin{split}
\frac{1}{\varepsilon _{n} K_{n}^{1+\tau }} = \Big [ \log \Big ( T^{
\frac{1}{2}} n^{\frac{1+\tau }{4(2+\tau )}} \Big ) \Big ]^{-\zeta -
\frac{1}{2}}.
\end{split}
\end{equation*}
Moreover,
\begin{equation*}
\begin{split}
\sqrt{ \frac{T}{\varepsilon _{n}}} \exp \left \lbrace - \frac{n
\varepsilon _{n}^{2}}{576(m+1)^{d} K_{n}^{2}} \right \rbrace ={}
&
\Big ( T^{\frac{1}{2}} n^{\frac{1+\tau }{4(2+\tau )}} \Big )^{1-\frac{r
_{n}}{576(m+1)^{d}}}
y_{n},
\end{split}
\end{equation*}
with $r_{n} = \Big [ \log (T^{1/2} n^{\frac{1+\tau }{4(2+\tau )}})
\Big ]^{2\zeta }$. Hence, the assertion of the theorem follows.
\end{proof}

\begin{rem}
\begin{enumerate}[(a)]%
\item
Fix $T > 0$. Then, provided the assumptions of
Theorem~\ref{theo:clt_univariate} are satisfied,
Theorem~\ref{theo:conv_rate} states that
\begin{equation*}
\max \Big \{ \sup _{t \in [-T,T]} |\hat{\psi }(t) - \psi (t)|,
\sup _{t \in [-T,T]}
|\hat{\theta }(t) - \theta (t)| \Big \} =
\mathcal{O}_{\mathbb{P}} (\varepsilon _{n}),
\end{equation*}
as $n \to \infty $, where $\mathcal{O}_{\mathbb{P}}$ denotes the
probabilistic order of convergence.
\item
\textit{For large $n$} in Theorem~\ref{theo:conv_rate} is understood in
the following sense: for any fixed $m$, there exists $n_{0} = n_{0}(m)$
such that the bound holds for all $n \geq n_{0}$. Of course, the
function $m \mapsto n_{0}(m)$ is increasing.
\end{enumerate}
\end{rem}

The following corollary is an immediate consequence of
Theorem~\ref{theo:conv_rate}.

\begin{cor}
\label{cor:sup_bn}
Let the assumptions of Theorem~\ref{theo:clt_univariate} be satisfied.
Then
\begin{equation*}
\lim \limits _{n \to \infty } \mathbb{P}\Big (
\sup _{t \in [-b_{n}^{-1},b_{n}^{-1}]} |\hat{\psi }(t) - \psi (t)|
\geq c b_{n}^{ \frac{1}{2} - \varepsilon } \Big ) = 0
\end{equation*}
for any constant $c > 0$.
\end{cor}

\begin{proof}
Fix $c > 0$ and assume that $b_{n} = n^{-\frac{1}{1-2\varepsilon
}} \Big ( \log n \Big )^{\eta + \frac{1}{1-2\varepsilon }} $, without loss of generality. Since
$b_{n} \to 0$ as $n \to \infty $, there exists $n_{0} \in \mathbb{N}$
such that $c b_{n}^{ \frac{1}{2}-\varepsilon } < \min \{ 1, \frac{1}{4
b_{n}} \}$ for all $n \geq n_{0}$. Taking $\varepsilon _{n} = c b_{n}
^{\frac{1}{2}-\varepsilon }$ and $T = b_{n}^{-1}$ in
Lemma~\ref{lem:conv_rate}, it follows that
\begin{align*}
& \mathbb{P}\Big ( \sup _{t \in [-b_{n}^{-1},b_{n}^{-1}]} |\hat{\psi }(t)
- \psi (t)| \geq c b_{n}^{ \frac{1}{2} - \varepsilon } \Big )
\\
\leq{}
&
\mathbb{P}\Big ( \sup _{t \in [-b_{n}^{-1},b_{n}^{-1}]} n^{-1}|S
_{W}(\xi ^{(1)}(t))| \geq c b_{n}^{ \frac{1}{2} - \varepsilon } \Big )
\\
& + \mathbb{P}\Big ( \sup _{t \in [-b_{n}^{-1},b_{n}^{-1}]} n^{-1}|S
_{W}(\xi ^{(2)}(t))| \geq c b_{n}^{ \frac{1}{2} - \varepsilon } \Big )
\\
\leq{}
& 2 \tilde{C} c^{-\frac{1}{2}} b_{n}^{\frac{2\varepsilon -3}{4}}
\exp \left \lbrace -
\frac{c^{2} n b_{n}^{1-2\varepsilon }}{160(m+1)^{d}} \right \rbrace\vadjust{\goodbreak}
\end{align*}
\endgroup
for all $n \geq n_{0}$ and some constant $\tilde{C} > 0$. Hence,
\begin{equation*}
\begin{split}
& \mathbb{P}\Big ( \sup _{t \in [-b_{n}^{-1},b_{n}^{-1}]} |\hat{\psi }(t)
- \psi (t)| > c b_{n}^{ \tilde{ \delta }/2 } \Big )
\\
\leq{}
& \check{C} n^{\frac{3-2\varepsilon }{4(1-2\varepsilon )} - \frac{c
^{2}(\log n)^{\eta (1-2\varepsilon )}}{160(m+1)^{d}}}
\Big ( \log n
\Big )^{-\frac{3-2\varepsilon }{4(1-2\varepsilon )} (1+\eta (1-2\varepsilon
))} \to 0,
\\
\end{split}
\end{equation*}
as $n \to \infty $, where $\check{C} = 2 \tilde{C} c^{-\frac{1}{2}}$.
\end{proof}

\begin{cor}
\label{cor:kappa_n}
Let $\gamma = 0$ and suppose that the assumptions of
Theorem~\ref{theo:clt_univariate} are satisfied. Moreover, let
$\kappa _{n} = 2 \Big ( \log n \Big )^{
\frac{1+\eta (1+2\varepsilon )}{2}}$. Then,
\begin{equation*}
\lim \limits _{n \to \infty } \mathbb{P}\Big ( \frac{|\psi (t)|}{|
\tilde{\psi }(t)|} \geq \kappa _{n} \ \text{for some} \ t \in [-b_{n}
^{-1}, b_{n}^{-1}] \Big ) = 0.
\end{equation*}
\end{cor}

\begin{proof}
By Lemma~\ref{lem:important_lemma}, (2), $\frac{1}{|\psi (x)|} \leq c
(1+|x|)^{\frac{1}{2} - \varepsilon }$ for some constant $c > 0$; hence,
there exists $n_{0} \in \mathbb{N}$ such that
%
\begin{equation}
\label{eq:lowerphi1}
\inf _{ t \in [-b_{n}^{-1}, b_{n}^{-1}] } |\psi (t)| \geq c^{-1} (1+|b
_{n}^{-1}|)^{\varepsilon - \frac{1}{2}} \geq c^{-1} b_{n}^{
\frac{1}{2} - \varepsilon }
\end{equation}
for all $n \geq n_{0}$. We first show that the probabilities of the events
\begin{equation*}
A_{n} := \Big \{ |\hat{\psi }(t)| < n^{-1/2} \ \text{for some} \ t
\in [-b_{n}^{-1}, b_{n}^{-1}] \Big \}
\end{equation*}
tend to $0$ as $n \to \infty $: By~\eqref{eq:psi_bigger_n_half},
$t \in [-b_{n}^{-1}, b_{n}^{-1}]$ implies $|\psi (t)| > 2n^{-1/2}$, for
all $n \geq n_{1}$ and some $n_{1} \in \mathbb{N}$. Set $\tilde{n} =
\max \{ n_{0}, \ n_{1} \}$. Then
\begin{equation*}
\begin{split}
\mathbb{P}(A_{n}) \leq{}
& \mathbb{P}\Big ( |\hat{\psi }(t) - \psi (t)| >
|\psi (t)| - n^{-1/2} \ \text{for some} \ t \in [-b_{n}^{-1}, b_{n}
^{-1}] \Big )
\\
\leq{}
& \mathbb{P}\Big ( |\hat{\psi }(t) - \psi (t)| > \frac{1}{2}|
\psi (t)| \ \text{for some} \ t \in [-b_{n}^{-1}, b_{n}^{-1}] \Big )
\\
\leq{}
& \mathbb{P}\Big ( |\hat{\psi }(t) - \psi (t)| > \frac{1}{2c} b
_{n}^{\frac{1}{2} - \varepsilon } \ \text{for some} \ t \in [-b_{n}
^{-1}, b_{n}^{-1}] \Big ),
\end{split}
\end{equation*}
for all $n \geq \tilde{n}$, where the last inequality follows
from~\eqref{eq:lowerphi1}. Hence, by Corollary~\ref{cor:sup_bn},
$\lim _{n \to \infty } \mathbb{P}(A_{n}) = 0$.

Suppose $\kappa _{n} = 2 \Big ( \log n \Big )^{\frac{1+\eta (1-2\varepsilon
)}{2}}$. Then we find that
\begin{align*}
& \mathbb{P}\Big ( \sup _{t \in [-b_{n}^{-1}, b_{n}^{-1}]} \frac{|
\psi (t)|}{|\tilde{\psi }(t)|} \geq \kappa _{n} \Big )
\\
={}
& \mathbb{P}\Big ( \sup _{t \in [-b_{n}^{-1}, b_{n}^{-1}]} \frac{|
\psi (t)|}{|\hat{\psi }(t)|} \geq \kappa _{n}, \ \inf _{t \in [-b_{n}^{-1}, b_{n}^{-1}]} |\hat{\psi }(t)| \geq n^{-1/2}
\Big )
\\
& + \mathbb{P}\Big ( \sup _{t \in [-b_{n}^{-1}, b_{n}^{-1}]} \frac{|
\psi (t)|}{|\tilde{\psi }(t)|} \geq \kappa _{n}, \ \inf _{t \in [-b_{n}^{-1}, b_{n}^{-1}]} |\hat{\psi }(t)| < n^{-1/2}
\Big )
\\
\leq{}
& \mathbb{P}\Big ( \sup _{t \in [-b_{n}^{-1}, b_{n}^{-1}]} \frac{|
\psi (t) - \hat{\psi }(t)|}{|\hat{\psi }(t)|} \geq \kappa _{n} - 1, \ \inf _{t \in [-b_{n}^{-1}, b_{n}^{-1}]} |\hat{\psi }(t)| \geq n^{-1/2}
\Big )
\\
& + \mathbb{P}\Big ( \inf _{t \in [-b_{n}^{-1}, b_{n}^{-1}]} |
\hat{\psi }(t)| < n^{-1/2} \Big )
\\
\leq{}
& \mathbb{P}\Big ( \sup _{t \in [-b_{n}^{-1}, b_{n}^{-1}]} |
\psi (t) - \hat{\psi }(t)| \geq (\kappa _{n} - 1) n^{-1/2} \Big )
+
\mathbb{P}(A_{n}),
\end{align*}
for all $n \geq \tilde{n}$. Taking into account that for large $n$,
$(\kappa _{n} - 1) n^{-1/2} = 2b_{n}^{\frac{1}{2} - \varepsilon } - n
^{-1/2} \geq b_{n}^{\frac{1}{2} - \varepsilon }$, the assertion follows
by Corollary~\ref{cor:sup_bn}.
\end{proof}

Now we can give the proof for
Theorem~\ref{theo:remainder_negligible_theorem}.

\begin{proof}[Proof of Theorem~\ref{theo:remainder_negligible_theorem}]
First of all, observe that
\begin{align*}
R_{n} + \theta \frac{\psi - \hat{\psi }}{\psi ^{2}} \Eins_{|
\hat{\psi }| \leq n^{-1/2}} ={}
& \Big ( 1- \frac{\hat{\psi }}{\psi }
\Big )
\Big ( \frac{\hat{\theta } - \theta }{\tilde{\psi }} + \theta \frac{
\psi - \hat{\psi }}{\psi \tilde{\psi }} \Big )
\\
={}
& \frac{\psi }{\tilde{\psi }} \Big ( \frac{\psi - \hat{\psi }}{
\psi } \Big ) \Big ( \frac{\hat{\theta } - \theta }{\psi } + \theta \frac{
\psi - \hat{\psi }}{\psi ^{2}} \Big )
\\
={}
& \frac{\psi }{\tilde{\psi }} \Big ( \frac{\psi - \hat{\psi }}{
\psi } \Big ) \Big ( \frac{\hat{\theta } - \theta }{\psi } +
i \Big ( \frac{1}{
\psi } \Big )^{\prime } \Big ( \psi - \hat{\psi } \Big ) \Big ).
\end{align*}
Now, fix $\tilde{\gamma } > 0$ and let $\kappa _{n} = 2 \Big ( \log n
\Big )^{\frac{1+\eta (1-2\varepsilon )}{2}}$. Moreover, let
\begin{equation*}
M_{n} = \max \Big \{ \sup _{t \in [-b_{n}^{-1}, b_{n}^{-1}]} |
\hat{\psi }(t) - \psi (t)|, \ \sup _{t \in [-b_{n}^{-1}, b_{n}^{-1}]} |
\hat{\theta }(t) - \theta (t)| \Big \}.
\end{equation*}
By \textbf{(K3)}, $\sup _{x \in \mathbb{R}, \ n \in \mathbb{N}}|
\mathcal{F}_{+}[K_{b_{n}}](x)| \leq 2$; hence
\begin{equation*}
\begin{split}
& \mathbb{P}\Big ( E_{2} \geq \tilde{\gamma } \Big )
\\
={}
& \mathbb{P}\Big ( \sqrt{n} \Big < \mathcal{F}_{+}[\mathcal{G}^{-1
\ast }v], \Big \{ R_{n} + \theta \frac{\psi - \hat{\psi }}{\psi ^{2}}
\Eins_{\{|\hat{\psi }| \leq n^{-1/2}\}} \Big \} \mathcal{F}_{+}[K
_{b}] \Big >_{L^{2}(\mathbb{R}^{\times })} \geq \tilde{\gamma } \Big )
\\
\leq{}
& \mathbb{P}\Big ( \int _{-b_{n}^{-1}}^{b_{n}^{-1}} |\mathcal{F}
_{+}[\mathcal{G}^{-1 \ast }v](x)|
\frac{|\psi (x)|}{|\tilde{\psi }(x)|} \frac{|\hat{\psi }(x) - \psi (x)|}{|
\psi (x)|}
\\
& \times \Big ( \frac{|\hat{\theta }(x) - \theta (x)|}{|\psi (x)|} +
\Big | \Big ( \frac{1}{\psi } \Big )^{\prime }(x) \Big |
|\psi (x) -
\hat{\psi }(x)| \Big ) d x \geq \frac{\tilde{\gamma }}{2} n^{-
\frac{1}{2}} \Big )
\\
\leq{}
& \mathbb{P}\Big ( M_{n} b_{n}^{\frac{1}{2} - \varepsilon }
\int _{-b_{n}^{-1}}^{b_{n}^{-1}} \frac{|\mathcal{F}_{+}[\mathcal{G}
^{-1 \ast }v](x)|}{|\psi (x)|}
\Big ( \frac{1}{|\psi (x)|} + \Big |
\Big ( \frac{1}{\psi } \Big )^{\prime }(x) \Big | \Big ) d x
\geq \frac{
\tilde{\gamma }}{2} n^{-\frac{1}{2}} \kappa _{n}^{-1} \Big )
\\
& + \mathbb{P}\Big ( \sup _{t \in [-b_{n}^{-1}, b_{n}^{-1}]} \frac{|
\psi (t)|}{|\tilde{\psi }(t)|} > \kappa _{n} \Big )
+ \mathbb{P}\Big (
\sup _{t \in [-b_{n}^{-1}, b_{n}^{-1}]} |\psi (t) - \hat{\psi }(t)| > b
_{n}^{\frac{1}{2} - \varepsilon } \Big ).
\end{split}
\end{equation*}
Since by Lemma~\ref{lem:important_lemma}, (3), $\Big ( \frac{1}{\psi }
\Big )^{\prime } \in L^{\infty }(\mathbb{R})$, there is a constant
$\tilde{c} > 0$ such that
\begin{equation*}
\begin{split}
& \mathbb{P}\Big ( M_{n} b_{n}^{\frac{1}{2} - \varepsilon }
\int _{-b_{n}^{-1}}^{b_{n}^{-1}} \frac{|\mathcal{F}_{+}[\mathcal{G}
^{-1 \ast }v](x)|}{|\psi (x)|}
\Big ( \frac{1}{|\psi (x)|} + \Big |
\Big ( \frac{1}{\psi } \Big )^{\prime }(x) \Big | \Big ) dx
\geq \frac{
\tilde{\gamma }}{2} n^{-\frac{1}{2}} \kappa _{n}^{-1} \Big )
\\
\leq{}
& \mathbb{P}\Big ( M_{n} \int _{-b_{n}^{-1}}^{b_{n}^{-1}} \frac{|
\mathcal{F}_{+}[\mathcal{G}^{-1 \ast }v](x)|}{|\psi (x)|^{2}} dx
\geq \frac{\tilde{\gamma }}{2 \tilde{c}} n^{-\frac{1}{2}} \kappa _{n}
^{-1} b_{n}^{\varepsilon - \frac{1}{2}} \Big ).
\end{split}
\end{equation*}
Moreover, by Definition~\ref{def:admissible_function}, (iii), we have
for some $\check{c} > 0$,
\begin{equation*}
\begin{split}
& \mathbb{P}\Big ( M_{n} \int _{-b_{n}^{-1}}^{b_{n}^{-1}} \frac{|
\mathcal{F}_{+}[\mathcal{G}^{-1 \ast }v](x)|}{|\psi (x)|^{2}} dx
\geq \frac{\tilde{\gamma }}{2 \tilde{c}} n^{-\frac{1}{2}} \kappa _{n}
^{-1} b_{n}^{\varepsilon - \frac{1}{2}} \Big )
\\
\leq{}
& \mathbb{P}\Big ( M_{n} \int _{-b_{n}^{-1}}^{b_{n}^{-1}} \frac{(1+x
^{2})^{-1+\varepsilon }}{|\psi (x)|^{2}} (1+x^{2})^{1-\varepsilon -
\xi /2} dx
\geq \frac{\tilde{\gamma }}{2 \tilde{c} \check{c}} n^{-
\frac{1}{2}} \kappa _{n}^{-1} b_{n}^{\varepsilon - \frac{1}{2}} \Big )
\\
\leq{}
& \mathbb{P}\Big ( M_{n} \geq \frac{\tilde{\gamma }}{2 \tilde{c}
\check{c}} n^{-\frac{1}{2}} \kappa _{n}^{-1} b_{n}^{\frac{3}{2} -
\varepsilon - \xi }
\Big \| \frac{(1+\cdot ^{2})^{-
\frac{\varepsilon -1}{2}}}{\psi } \Big \|_{L^{2}(\mathbb{R})}^{-1}
\Big ),
\end{split}
\end{equation*}
where the last line follows because $\frac{(1+\, \cdot \,^{2})^{-(
\varepsilon -1)/2}}{\psi } \in L^{2}(\mathbb{R})$ by
Proposition~\ref{lem:properties_transferred}, (c). Now, for $n$
sufficiently large, we obtain
\begin{equation*}
\begin{split}
n^{-\frac{1}{2}} \kappa _{n}^{-1} b_{n}^{\frac{3}{2} - \varepsilon -
\xi } ={}
& \frac{1}{2} n^{ \frac{\xi - 1}{1-2\varepsilon } - 1 }
\big ( \log n \big )^{ (1-\xi ) \big ( \eta + \frac{1}{1-2\varepsilon }
\big ) }
\\
>{}
& n^{-\frac{1+\tau }{2(2+\tau )}}
\Big [ \log \Big ( b_{n}^{-
\frac{1}{2}} n^{\frac{1+\tau }{4(2+\tau )}} \Big ) \Big ]^{\eta +
\frac{1}{2}}.
\end{split}
\end{equation*}
Indeed, by Definition~\ref{def:admissible_function}, (iii), we have
\begin{equation*}
\frac{\xi - 1}{1-2\varepsilon } - 1 > \frac{1}{1-2\varepsilon } - \frac{1+
\tau }{2(2+\tau )} > - \frac{1+\tau }{2(2+\tau )}.
\end{equation*}
Hence, we conclude by Theorem~\ref{theo:conv_rate},
Corollary~\ref{cor:sup_bn} and Corollary~\ref{cor:kappa_n} that
\begin{equation*}
\mathbb{P}\Big ( E_{2} \geq \tilde{\gamma } \Big ) \to 0, \quad
\text{as} \ n \to \infty ,
\end{equation*}
for any $\tilde{\gamma } > 0$.
\end{proof}

\subsection{Neglecting the drift $\gamma $}%
\label{sec:negelcting_the_drift}
It remains to show that the result of Theorem~\ref{theo:clt_univariate}
still holds true if $\gamma $ is assumed to be arbitrary. For this
purpose, consider the sample $(\tilde{Y}_{j})_{j \in W}$ given by
$\tilde{Y}_{j} = Y_{j} - \gamma $, $j \in W$. Moreover, let
$\psi _{\ast }(t) = \mathbb{E}[ e^{i t \tilde{Y}_{0}} ]$ be the
characteristic function\index{characteristic function} of $\tilde{Y}_{0}$ and write $\hat{\psi }_{
\ast }$ for its empirical counterpart, i.e. $\hat{\psi }_{\ast }(t) =
\frac{1}{n} \sum _{j\in W} e^{i t \tilde{Y}_{j}}$. Then, with the
notation
\begin{equation*}
\frac{1}{\tilde{\psi }_{\ast }(t)} :=
\frac{1}{\hat{\psi }_{\ast }(t)}\Eins_{\{|\hat{\psi }_{\ast }(t)|
> n^{-1/2}\}} =
e^{it\gamma } \frac{1}{\tilde{\psi }(t)},
\end{equation*}
we have for any $t \in \mathbb{R}$,
%
\begin{equation}
\label{eq:discard_gamma}
\begin{split}
\frac{\hat{\theta }_{\ast }(t)}{\tilde{\psi }_{\ast }(t)} = \frac{
\hat{\theta }(t)}{\tilde{\psi }(t)} - \gamma \Eins_{\{|
\hat{\psi }_{\ast }(t)| > n^{-1/2}\}}
\quad \text{as well as} \quad \frac{
\theta _{\ast }(t)}{\psi _{\ast }(t)} = \frac{\theta (t)}{\psi (t)} -
\gamma ,
\end{split}
\end{equation}
where $\theta _{\ast }(t) = \mathbb{E}[ \tilde{Y}_{0} e^{i t \tilde{Y}
_{0}} ]$ and $\hat{\theta }_{\ast }(t) = \frac{1}{n} \sum _{j\in W}
\tilde{Y}_{j} e^{i t \tilde{Y}_{j}}$. For any $v \in \mathcal{U}(
\xi ,\beta _{2})$, consider the decomposition
\begin{equation*}
\begin{split}
\sqrt{n} (\hat{L}_{W} v - \mathcal{L}v) ={}
& \frac{\sqrt{n}}{2
\pi } \Big < \mathcal{F}_{+}[\mathcal{G}^{-1 \ast }v], \frac{
\hat{\theta }}{\tilde{\psi }} \mathcal{F}_{+}[K_{b_{n}}] - \frac{
\theta }{\psi } \Big >_{L^{2}(\mathbb{R})}
\\
={}
&\frac{\sqrt{n}}{2\pi } \Big < \mathcal{F}_{+}[\mathcal{G}^{-1
\ast }v], \frac{\hat{\theta }_{\ast }}{\tilde{\psi }_{\ast }}
\mathcal{F}_{+}[K_{b_{n}}] - \frac{\theta _{\ast }}{\psi _{\ast }}
\Big >_{L^{2}(\mathbb{R})}
\\
& + \frac{\sqrt{n}}{2\pi } \Big < \mathcal{F}_{+}[\mathcal{G}^{-1
\ast }v], \gamma \Eins_{\{|\hat{\psi }_{\ast }| > n^{-1/2}\}}
\mathcal{F}_{+}[K_{b_{n}}] - \gamma \Big >_{L^{2}(\mathbb{R})}.
\end{split}
\end{equation*}
As $W$ is regularly growing\index{regularly growing} to infinity, the first summand on the
right-hand side of the last equation tends to a Gaussian random
variable\index{Gaussian random variable} since $\psi _{\ast }$ is an infinitely divisible characteristic
function\index{infinitely divisible characteristic functions} without drift component. For the second summand, we find that
\begin{equation*}
\begin{split}
& \frac{\sqrt{n}}{2\pi } \Big < \mathcal{F}_{+}[\mathcal{G}^{-1
\ast }v], \gamma \Eins_{\{|\hat{\psi }_{\ast }| > n^{-1/2}\}}
\mathcal{F}_{+}[K_{b_{n}}] - \gamma \Big >_{L^{2}(\mathbb{R})}
\\
={}
& \frac{\sqrt{n} \gamma }{2\pi } \Big < \mathcal{F}_{+}[
\mathcal{G}^{-1 \ast }v], \mathcal{F}_{+}[K_{b_{n}}] - 1 \Big >_{L^{2}(
\mathbb{R})}
\\
& - \frac{\sqrt{n} \gamma }{2\pi } \Big < \mathcal{F}_{+}[
\mathcal{G}^{-1 \ast }v], \Eins_{\{|\hat{\psi }_{\ast }| \leq n
^{-1/2}\}} \mathcal{F}_{+}[K_{b_{n}}] \Big >_{L^{2}(\mathbb{R})}.
\end{split}
\end{equation*}
Hence, by \textbf{(K3)} and Definition~\ref{def:admissible_function},
(iii), we obtain
\begin{equation*}
\begin{split}
& \sqrt{n} \mathbb{E}\Big | \Big < \mathcal{F}_{+}[\mathcal{G}^{-1
\ast }v], \mathcal{F}_{+}[K_{b_{n}}] - 1 \Big >_{L^{2}(\mathbb{R})}
\Big |
\\
\leq{}
& \int _{\mathbb{R}} \sqrt{n} |\mathcal{F}_{+}[\mathcal{G}^{-1
\ast }v](x)| |1 - \mathcal{F}_{+}[K_{b_{n}}](x)| dx
\lesssim b_{n}
\sqrt{n} \|\mathcal{G}^{-1 \ast }v\|_{H^{1}(\mathbb{R})},
\end{split}
\end{equation*}
where the last term tends to zero as $n \to \infty $, since
$b_{n} = o(n^{-1/2})$. Moreover,
\begin{equation*}
\begin{split}
& \sqrt{n} \mathbb{E}\Big |\Big < \mathcal{F}_{+}[\mathcal{G}^{-1
\ast }v], \Eins_{\{|\hat{\psi }_{\ast }| \leq n^{-1/2}\}}
\mathcal{F}_{+}[K_{b_{n}}] \Big >_{L^{2}(\mathbb{R})} \Big |
\\
\leq{}
& 2 \sqrt{n} \int _{-b_{n}^{-1}}^{b_{n}^{-1}} |\mathcal{F}_{+}[
\mathcal{G}^{-1 \ast }v](x)| \mathbb{P}\Big ( |\hat{\psi }_{\ast }(x)|
\leq n^{-\frac{1}{2}} \Big ) dx.
\end{split}
\end{equation*}
Taking into account that $|\psi (x)| = |\hat{\psi }_{\ast }(x)|$,
relation~\eqref{eq:propb_estimation} with $p=1/2$ yields
\begin{equation*}
\sqrt{n} |\mathcal{F}_{+}[\mathcal{G}^{-1 \ast }v](x)| \mathbb{P}
\Big ( |\hat{\psi }_{\ast }(x)|\leq n^{-\frac{1}{2}} \Big ) \Eins
_{ [-b_{n}^{-1}, b_{n}^{-1}] }(x)
\leq \frac{|\mathcal{F}_{+}[
\mathcal{G}^{-1 \ast }v](x)|}{|\psi (x)|}.
\end{equation*}
Applying again~\eqref{eq:propb_estimation} with $p=1$ implies
\begin{equation*}
\sqrt{n} |\mathcal{F}_{+}[\mathcal{G}^{-1 \ast }v](x)| \mathbb{P}
\Big ( |\hat{\psi }_{\ast }(x)|\leq n^{-\frac{1}{2}} \Big ) \Eins
_{ [-b_{n}^{-1}, b_{n}^{-1}] }(x)
\leq n^{-\frac{1}{2}} \frac{|
\mathcal{F}_{+}[\mathcal{G}^{-1 \ast }v](x)|}{|\psi (x)|^{2}} \to 0,
\end{equation*}
as $n \to \infty $; thus, by dominated convergence, we have
\begin{equation*}
\sqrt{n} \mathbb{E}\Big |\Big < \mathcal{F}_{+}[\mathcal{G}^{-1
\ast }v], \Eins_{\{|\hat{\psi }_{\ast }| \leq n^{-1/2}\}}
\mathcal{F}_{+}[K_{b_{n}}] \Big >_{L^{2}(\mathbb{R})} \Big |
\to 0,
\quad \text{as} \ n \to \infty .
\end{equation*}
All in all, this shows that Theorem~\ref{theo:clt_univariate} holds for
any fixed $\gamma \in \mathbb{R}$.


\begin{appendix}
\section{Appendix}%
\label{sec:Appendix}

\subsection{Proof of Lemma~\ref{lem:properties_transferred}}
%
%
\begin{enumerate}[(a)]%
\item
Minkowski's integral inequality together with
formula~\eqref{eq:levy-characterisitic-of-field} yields
\begin{equation*}
\| uv_{1} \|_{L^{k}(\mathbb{R})} \leq \| uv_{0} \|_{L^{k}(\mathbb{R})}
\int _{\supp(f)} |f(s)|^{1/k} ds, \quad k=1,2.
\end{equation*}
The right-hand side in the last inequality is finite by
Assumption~\ref{ass:basic_assumptions}, (1) and (2); hence $uv_{1}
\in L^{1}(\mathbb{R}) \cap L^{2}(\mathbb{R})$. In particular,
$\mathbb{R}\ni x \mapsto \mathcal{F}_{+}[uv_{1}](x) =\break  \int _{
\mathbb{R}} e^{itx} (uv_{1})(t) dt$ is well-defined. Using again formula
\eqref{eq:levy-characterisitic-of-field} together with Fubini's theorem
and a simple integral substitution yields~\eqref{eq:Fouriertrafo_uv_1}.
\item
The triangle inequality followed by a simple integral substitution shows
that
\begin{equation*}
\int _{\mathbb{R}} |x|^{1+\tau } |(uv_{1})|(x) dx \leq \|f\|_{ L^{2+
\tau }(\mathbb{R}) } \int _{\mathbb{R}} |x|^{1+\tau } |(uv_{0})(x)| dx
< \infty .
\end{equation*}%
\item
The proof of Theorem 3.10 in~\cite{GlRothSpo017} yields that
$|\psi (x)|$ coincides with the inverse of the right-hand side
in~\eqref{eq:abs_psi_condition}. This shows part (c).\qed
\end{enumerate}
%

\subsection{Proof of Lemma~\ref{theo:upper_bound_estimation_error}}
%
Let $v \in \textup{Image}(\mathcal{G})$ be \xch{such}{sucht} that $\int _{\mathbb{R}}
\frac{|\mathcal{F}_{+}[\mathcal{G}^{-1 \ast }v]|(x)}{|\psi (x)|} dx <
\infty $. In order to prove the upper bound in
Theorem~\ref{theo:upper_bound_estimation_error}, decompose $
\mathbb{E}|\hat{\mathcal{L}}_{W} v - \mathcal{L}v|$ as follows:
\begin{equation*}
\begin{split}
\mathbb{E}|\hat{\mathcal{L}}_{W} v - \mathcal{L}v| \leq \underbrace{
\mathbb{E}\big | \left \langle \big ( \mathcal{G}_{n}^{-1 \ast } -
\mathcal{G}^{-1 \ast } \big ) v, \widehat{uv_{1}} \right \rangle _{L
^{2}(\mathbb{R})} \big |}_{\textbf{(I)}}
+ \mathbb{E}\underbrace{
\big | \left \langle \mathcal{G}^{-1 \ast } v, \widehat{uv_{1}} - uv
_{1} \right \rangle _{L^{2}(\mathbb{R})} \big |}_{\textbf{(II)}}.
\end{split}
\end{equation*}
We estimate parts (I) and (II) seperately. Using the isometry property\index{isometry property}
of $\mathcal{F}_{+}$, we obtain
\begin{equation*}
\begin{split}
\textbf{(I)}
\leq \frac{1}{2\pi } \int _{\mathbb{R}} | \mathcal{F}_{+}
[ (\mathcal{G}_{n}^{-1 \ast } - \mathcal{G}^{-1 \ast }) v ](x) |
\mathbb{E}| \mathcal{F}_{+}[\widehat{uv_{1}}](x) | dx.
\end{split}
\end{equation*}
Furthermore, since $\widehat{u v_{1}} = \mathcal{F}_{+}^{-1}[\frac{
\hat{\theta }(x)}{\tilde{\psi }(x)} \mathcal{F}_{+}[K_{b}]]$,
stationarity of $Y$ yields for any $x \in \mathbb{R}$,
\begin{align*}
\mathbb{E}| \mathcal{F}_{+}[\widehat{uv_{1}}](x) |
&= | \mathcal{F}
_{+}[K_{b}](x) | \mathbb{E}\left | \frac{\sum _{t \in W} Y_{t}}{n
\hat{\psi }(x)} \Eins_{\{ |\hat{\psi }(x)| > n^{-1/2} \}} \right |\\
&\leq n^{1/2} | \mathcal{F}_{+}[K_{b}](x) | \mathbb{E}|Y_{0}|.
\end{align*}
Hence, by the \xch{Cauchy--Schwarz}{Cauchy-Schwart} inequality we obtain that
\begin{equation*}
\begin{split}
\textbf{(I)}
& \leq \frac{n^{1/2} \mathbb{E}|Y_{0}|}{2\pi } \big \|
\mathcal{F}_{+} [ (\mathcal{G}_{n}^{-1 \ast } - \mathcal{G}^{-1
\ast }) v ] \big \|_{L^{2}(\mathbb{R})}
\| \mathcal{F}_{+}[K_{b}] \|
_{L^{2}(\mathbb{R})}
\\
& \leq \frac{S \mathbb{E}|Y_{0}|}{\sqrt{\pi }} \Big ( \frac{n}{b}
\Big )^{1/2} \big \| (\mathcal{G}_{n}^{-1 \ast } - \mathcal{G}^{-1
\ast }) v \big \|_{L^{2}(\mathbb{R})},
\end{split}
\end{equation*}
where the last line follows from \textbf{(K2)} and again by applying the
isometry property\index{isometry property} of $\mathcal{F}_{+}$. For the second part, we find
that
\begin{equation*}
\begin{split}
\textbf{(II)} ={}
& \frac{1}{2 \pi } \int _{\mathbb{R}} | \mathcal{F}
_{+}[\mathcal{G}^{-1 \ast }v](x) | \mathbb{E}\left | \frac{
\hat{\theta }(x)}{\tilde{\psi }(x)} \mathcal{F}_{+}[K_{b}](x) -
\frac{
\theta (x)}{\psi (x)} \right | \:\mathrm{d}x
\\
\leq{}
& \frac{1}{2 \pi }\underbrace{\int _{\mathbb{R}} | \mathcal{F}
_{+}[\mathcal{G}^{-1 \ast }v](x) | \mathbb{E}\left | \frac{
\hat{\theta }(x)}{\tilde{\psi }(x)} - \frac{\theta (x)}{\psi (x)}
\right | |\mathcal{F}_{+}[K_{b}](x)| \:\mathrm{d}x}_{\textbf{(III)}}
\\
& + \frac{1}{2 \pi } \left \langle |\mathcal{F}_{+}[\mathcal{G}^{-1
\ast }v]|, |\mathcal{F}_{+}[uv_{1}]| |1 - \mathcal{F}_{+}[K_{b}]| \right \rangle _{L^{2}(\mathbb{R})},
\end{split}
\end{equation*}
where the identity $|\mathcal{F}_{+}[uv_{1}]|(x) = \left | \frac{
\theta (x)}{\psi (x)} \right |$ was used in the last line. Hence, it
remains to bound expression \textbf{(III)}. Indeed, applying the triangle
inequality followed by\break  the \xch{Cauchy--Schwarz}{Cauchy-Schwart} inequality and the bounds
in~\cite[Lemma 8.1 and 8.3]{KaRoSpoWalk18} yields
\begin{equation*}
\begin{split}
\textbf{(III)} \leq{}
& \int _{\mathbb{R}} |\mathcal{F}_{+}[K_{b}](x)| |
\mathcal{F}_{+}[\mathcal{G}^{-1 \ast }v](x)| \mathbb{E}|\hat{\theta }(x)-
\theta (x)| \Big | \frac{1}{\tilde{\psi }(x)} - \frac{1}{\psi (x)}
\Big | dx
\\
& + \int _{\mathbb{R}} |\mathcal{F}_{+}[K_{b}](x)| |\mathcal{F}_{+}[
\mathcal{G}^{-1 \ast }v](x)| |\theta (x)| \mathbb{E}\Big | \frac{1}{
\tilde{\psi }(x)} - \frac{1}{\psi (x)} \Big | dx
\\
& + \int _{\mathbb{R}} |\mathcal{F}_{+}[K_{b}](x)| |\mathcal{F}_{+}[
\mathcal{G}^{-1 \ast }v](x)| \frac{\mathbb{E}|\hat{\theta }(x) -
\theta (x)|}{|\psi (x)|} dx
\\
\leq{}
& S \left [ \int _{\mathbb{R}} |\mathcal{F}_{+}[\mathcal{G}^{-1
\ast }v](x)| \sqrt{\mathbb{E}|\hat{\theta }(x)-\theta (x)|^{2}} \sqrt{
\mathbb{E}\Big | \frac{1}{\tilde{\psi }(x)} - \frac{1}{\psi (x)}
\Big |^{2}} dx \right .
\\
& + \left . \int _{\mathbb{R}} \frac{|\mathcal{F}_{+}[\mathcal{G}^{-1
\ast }v](x)|}{|\psi (x)|} |\mathcal{F}_{+}[uv_{1}](x)| \sqrt{
\mathbb{E}\Big | \frac{1}{\tilde{\psi }(x)} - \frac{1}{\psi (x)}
\Big |^{2}} dx \right .
\\
& + \left . \int _{\mathbb{R}} \frac{|\mathcal{F}_{+}[\mathcal{G}^{-1
\ast }v](x)}{|\psi (x)|}
\sqrt{\mathbb{E}|\hat{\theta }(x)-\theta (x)|^{2}} dx \right ]
\\
\leq{}
& S \left [ \frac{c_{1}}{n^{1/2}}
\sqrt{\mathbb{E}|Y_{0}|^{2}} \int _{\mathbb{R}} \frac{|\mathcal{F}_{+}[
\mathcal{G}^{-1 \ast }v](x)|}{|\psi (x)|} dx \right .
\\
& + \left . \frac{c_{2}}{n^{1/2}} \int _{\mathbb{R}} \frac{|
\mathcal{F}_{+}[\mathcal{G}^{-1 \ast }v](x)|}{|\psi (x)|} |
\mathcal{F}_{+}[uv_{1}](x)| dx \right .
\\
& + \left . \frac{c_{3}}{n^{1/2}} \sqrt{\mathbb{E}|Y_{0}|^{2}}
\int _{\mathbb{R}} \frac{|\mathcal{F}_{+}[\mathcal{G}^{-1 \ast }v](x)}{|
\psi (x)|} dx \right ],
\end{split}
\end{equation*}
with constants $c_{1}$, $c_{2}$, $c_{3} > 0$. Hence, by integrability
of $uv_{1}$ it follows
\begin{equation*}
\textbf{(III)} \leq \frac{c \cdot S}{\sqrt{n}} \left ( \sqrt{
\mathbb{E}|Y_{0}|^{2}} + \|uv_{1}\|_{L^{1}(\mathbb{R}^{\times })}
\right ) \int _{\mathbb{R}^{\times }} \frac{|\mathcal{F}_{+}[
\mathcal{G}^{-1 \ast }v](x)|}{|\psi (x)|} dx
\end{equation*}
for some constant $c > 0$. This finishes the proof.\qed

\subsection{Proof of Theorem~\ref{cor:order_of_convergence}}
%
Using Assumption~\ref{ass:basic_assumptions}, (4), \textbf{(K3)} and
$\mathcal{F}_{+}[\mathcal{G}^{-1 \ast }] \in L^{1}(\mathbb{R})$ we find
that
\begin{equation*}
\begin{split}
\left \langle |\mathcal{F}_{+}[\mathcal{G}^{-1 \ast } v] |, |
\mathcal{F}_{+}[uv_{1}]| |1-\mathcal{F}_{+}[K_{b}]| \right \rangle _{L^{2}(\mathbb{R})}
\lesssim{}
& \min \{ 1, b_{n} \} \int _{\mathbb{R}}
|\mathcal{F}_{+}[\mathcal{G}^{-1 \ast } v] (x)| dx
\\
={}
& \mathcal{O}(b_{n}), \quad \text{as} \ n \to \infty .
\end{split}
\end{equation*}
Moreover, applying the same arguments as in the proof
of~\cite[Corollary 3.7]{GlRothSpo017}, we observe that
\begin{equation*}
\|\big ( \mathcal{G}_{n}^{-1 \ast } - \mathcal{G}^{-1 \ast } \big ) v\|
_{L^{2}(\mathbb{R})} \lesssim a_{n}^{\frac{\beta _{2}}{\beta _{1}} - 1}.
\end{equation*}
Hence, if $\gamma = 0$, the assertions of the theorem immediately follow
by the upper bound in Lemma~\ref{theo:upper_bound_estimation_error}.
Otherwise, if $\gamma \neq 0$, consider the sample $(\tilde{Y}_{j})_{j
\in W}$ defined in Section~\ref{sec:negelcting_the_drift}. Following the
computations there, one finds that on the right-hand side
of~\eqref{eq:upper_bound_uv_1} the additional term
\begin{equation*}
\frac{\gamma }{2\pi } \mathbb{E}\Big | \Big < \mathcal{F}_{+}[
\mathcal{G}^{-1 \ast }v], \mathcal{F}_{+}[K_{b_{n}}]
- 1 \Big >_{L^{2}(
\mathbb{R})} - \Big < \mathcal{F}_{+}[\mathcal{G}^{-1 \ast }v],
\Eins_{\{|\hat{\psi }_{\ast }| \leq n^{-1/2}\}} \mathcal{F}_{+}[K
_{b_{n}}] \Big >_{L^{2}(\mathbb{R})} \Big |
\end{equation*}
arises. Using $\mathcal{G}^{-1 \ast } v \in H^{1}(\mathbb{R})$,
$\mathcal{F}_{+}[\mathcal{G}^{-1 \ast }] \in L^{1}(\mathbb{R})$ and
\textbf{(K3)} yields that the latter expression can be estimated from
above by
\begin{equation*}
\frac{\gamma }{2\pi } \Big ( b_{n} \|\mathcal{G}^{-1 \ast } v\|_{H^{1}(
\mathbb{R})} + \frac{1}{\sqrt{n}} S \Big \| \frac{\mathcal{F}_{+}[
\mathcal{G}^{-1 \ast }v]}{\psi } \Big \|_{L^{1}(\mathbb{R})} \Big ).
\end{equation*}
This completes the proof.\qed

\subsection{Moment inequalities for $m$-dependent random fields}%
\label{subsec:moment_inequalities}
In this section, we sum up some moment inequalities that are quite
helpful for the proofs in Section~\ref{section:clt}.

We start with the following Bernstein type inequality that is due
to~\cite[p. 316]{Heinrich}.
%
\begin{theo}
\label{theo:bernstein}
Let $(X_{j})_{j \in V}$, $V \subset \mathbb{Z}^{d}$, be a centered
$m$-dependent random field satisfying $0 < \mathbb{E}X_{j}^{2} <
\infty $ and, for some $H > 0$,
%
\begin{equation}
\label{eq:bernstein_cond}
|\mathbb{E}X_{j}^{p}| \leq \frac{p!}{2} H^{p-2} \mathbb{E}X_{j}^{2},
\quad p \geq 3, \ j \in V.
\end{equation}
Then
\begin{equation*}
P(S_{V}(X) \geq x B_{V}) \leq
\begin{cases}
\exp \left ( -\frac{x^{2}}{4 (m+1)^{d} \rho _{V}} \right ),
& 0 \leq x
\leq \rho _{V} B_{V} / H,
\\
\exp \left ( -\frac{x B_{V}}{4 (m+1)^{d} H} \right ),
& x \geq \rho
_{V} B_{V} / H,
\end{cases}
\end{equation*}
where
\begin{equation*}
S_{V}(X) = \sum _{j \in V} X_{j}, \qquad B_{V}^{2} = \mathbb{E}S_{V}
^{2} \quad \text{and} \quad \rho _{V} = \sum _{j \in V} \mathbb{E}X_{j}
^{2} / B_{V}^{2}.
\end{equation*}
\end{theo}

The following lemma generalizes Lemma 8.1 in~\cite{KaRoSpoWalk18}.
It can be easily proven using the same arguments as there.

\begin{lem}
\label{lem:moment_bound_theta}
Let $(Y_{j})_{j \in Z^{d}}$ be a stationary $m$-dependent random field\index{random ! field}
satisfying $\mathbb{E}|Y_{0}|^{2q} < \infty $. Furthermore, let
$W \subset \mathbb{Z}^{d}$ be a finite subset, $n = \textup{card}(W)$,
and let $\hat{\theta }(u) = \frac{1}{n} \sum _{j \in W} Y_{j} e^{iuY
_{j}}$ and $\theta (u) = \mathbb{E}Y_{0} e^{iuY_{0}}$. Then
\begin{equation*}
\mathbb{E}|\hat{\theta }(u) - \theta (u)|^{2q} \leq \frac{C}{n^{q}}
\mathbb{E}|Y_{0}|^{2q},
\end{equation*}
where $C > 0$ is a constant.
\end{lem}

\begin{rem}
Clearly, applying the \xch{Cauchy--Schwarz}{Cauchy-Schwart} inequality,
Lemma~\ref{lem:moment_bound_theta} also yields a bound in case that
$q = 1/2$.
\end{rem}

\subsection{Asymptotic covariance of $m$-dependent random \xch{field}{fiels}}%
\label{subsec:asymptotic_variance}
%
\begin{lem}
\label{lem:conv_of_covariance_fct_cubes_version}
Let the sequence $(B_{n})_{n\in \mathbb{N}}$ be \xch{regularly}{regulary} growing\index{regularly growing} to
infinity. Moreover, let $ (X_{j})_{j \in \mathbb{Z}^{d}}$ be a
stationary $m$-dependent random field\index{random ! field} and suppose there are measurable functions
$g^{(1)}$, $g_{n}^{(1)}$, $g^{(2)}$, $g_{n}^{(2)}:\mathbb{R}\to
\mathbb{R}$, $n \in \mathbb{N}$, with the following properties:
\begin{enumerate}%
\item
$\mathbb{E}[g_{n}^{(k)}(X_{0})] = 0$ for all $n \in \mathbb{N}$,
$k=1,2$;
\item
$\mathbb{E}[g^{(k)}(X_{0})^{2}]$, $\mathbb{E}[g_{n}^{(k)}(X_{0})^{2}]
< \infty $, $k=1,2$, $n \in \mathbb{N}$;
\item
$\lim _{n \rightarrow \infty } \mathbb{E}[g_{n}^{(1)}(X_{0})g_{n}^{(2)}(X
_{k})] = \mathbb{E}[g^{(1)}(X_{0})g^{(2)}(X_{k})] =: \sigma _{k}$, for
any $k \in \mathbb{Z}^{d}$.
\end{enumerate}
Then
\begin{equation*}
\lim _{n\rightarrow \infty } \Cov \Big ( |B_{n}|^{-1/2} \sum _{j \in B
_{n}} g_{n}^{(1)}(X_{j}), \ |B_{n}|^{-1/2} \sum _{k \in B_{n}} g_{n}
^{(2)}(X_{k}) \Big )
=
\sum _{\substack{t \in \mathbb{Z}^{d}:
\\
\|t\|_{\infty }\leq m}} \sigma _{t}.
\end{equation*}
\end{lem}

\begin{proof}
We observe that
\begin{equation*}
\begin{split}
& \Cov \Big ( |B_{n}|^{-1/2} \sum _{j \in B_{n}} g_{n}^{(1)}(X_{j}), \ |B
_{n}|^{-1/2} \sum _{k \in B_{n}} g_{n}^{(2)}(X_{k}) \Big )
\\
={}
& \underbrace{\frac{1}{|B_{n}|} \sum _{j\in B_{n}} \sum _{k\in B_{n}}
\left ( \mathbb{E}\Big [ g_{n}^{(1)}(X_{j}), g_{n}^{(2)}(X_{k}) \Big ] -
\mathbb{E}\Big [ g^{(1)}(X_{j}), g^{(2)}(X_{k}) \Big ] \right )}_{=y
_{n}}
\\
& + \underbrace{\frac{1}{|B_{n}|} \sum _{j\in B_{n}} \sum _{k\in B_{n}}
\mathbb{E}\Big [ g^{(1)}(X_{j}), g^{(2)}(X_{k}) \Big ]}_{=z_{n}}.
\end{split}
\end{equation*}
Since $(X_{j})_{j \in \mathbb{Z}^{d}}$ is $m$-dependent and stationary,
and, since $(B_{n})_{n \in \mathbb{N}}$ is regularly growing\index{regularly growing} to
infinity, the same\vadjust{\goodbreak} computation as in the proof of Theorem 1.8
in~\cite[p.175]{Bulinski07} shows that
\begin{equation*}
\lim \limits _{n \to \infty } z_{n} =
\sum _{\substack{t \in \mathbb{Z}^{d}:
\\
\|t\|_{\infty }\leq m}} \sigma _{t}.
\end{equation*}
It remains to show that $\lim _{n \to \infty } y_{n} = 0$. Indeed,
the $m$-dependence and the property 1. yield
\begin{equation*}
\begin{split}
|y_{n}| \leq{}
& \frac{1}{|B_{n}|} \sum _{j \in B_{n}}
\sum _{k \in \mathbb{Z}^{d}} \Big | \mathbb{E}\Big [ g_{n}^{(1)}(X_{0}),
g_{n}^{(2)}(X_{k-j}) \Big ]
- \mathbb{E}\Big [ g^{(1)}(X_{0}), g^{(2)}(X
_{k-j}) \Big ] \Big |
\\
\leq{}
& \sum _{\substack{k \in \mathbb{Z}^{d}: \\ \|k\|_{\infty }
\leq m}} \Big | \mathbb{E}\Big [ g_{n}^{(1)}(X_{0}), g_{n}^{(2)}(X_{k})
\Big ]
- \mathbb{E}\Big [ g^{(1)}(X_{0}), g^{(2)}(X_{k}) \Big ] \Big |
\\
& \to 0 ,\quad \text{as} \ n \to \infty .\qedhere
\end{split}
\end{equation*}
\end{proof}
\end{appendix}

\begin{acknowledgement}
I would like to thank Evgeny Spodarev and Alexander Bulinski for their
fruitful discussions on the subject of this paper.
\end{acknowledgement}


\end{document}